\documentclass[11pt,a4paper]{article} 
\usepackage[T1]{fontenc}
\usepackage[english]{babel}
\usepackage[utf8]{inputenc}
\usepackage{amsfonts}
\usepackage{amsmath}
\usepackage{amssymb}
\usepackage{amsthm}
\usepackage{enumerate}
\usepackage{graphicx}
\usepackage{dsfont}
\usepackage{array}
\usepackage{url}
\usepackage{float}
\usepackage{tikz}

\author{Arielle MARC-ZWECKER}
\title{Triangle groups in the complex hyperbolic plane and special fibers of the momentum map in PU(2,1).}
\begin{document}
	\maketitle
	
	\theoremstyle{plain}
	\newtheorem{Theo}{Theorem}[section]
	\newtheorem*{Theo*}{Theorem}
	\newtheorem{Prop}[Theo]{Proposition}
	\newtheorem{Prob}[Theo]{Probl\`eme}
	\newtheorem{Lemm}[Theo]{Lemma}  
	\newtheorem{Coro}[Theo]{Corollary}
	\newtheorem{Propr}[Theo]{Propri\'et\'e}
	\newtheorem{Conj}[Theo]{Conjecture}
	\newtheorem*{Conj*}{Conjecture}
	\newtheorem{Aff}[Theo]{Affirmation}
	\newtheorem{thm}{Theorem}[section]
	\renewcommand{\thethm}{\empty{}} 
	\newtheorem{Def}{Definition}[section]
	\renewcommand{\theDef}{\empty{}}
	\newtheorem{conj}{Conjecture}[section]
	\renewcommand{\theconj}{\empty{}}
	
	\theoremstyle{definition}
	\newtheorem{Defi}[Theo]{Definition}
	\newtheorem*{Defi*}{Definition}
	\newtheorem{Exem}[Theo]{Example}
	\newtheorem{Nota}[Theo]{Notation}
	
	\theoremstyle{remark}
	\newtheorem{Rema}[Theo]{Remark}
	\newtheorem{NB}[Theo]{N.B.}
	\newtheorem{Comm}[Theo]{Commentaire}
	\newtheorem{question}[Theo]{Question}
	\newtheorem{exer}[Theo]{Exercice}

	\def\emptyset{\varnothing}
	
	\def\NN{{\mathbb N}}    
	\def\ZZ{{\mathbb Z}}     
	\def\RR{{\mathbb R}}    
	\def\QQ{{\mathbb Q}}    
	\def\CC{{\mathbb C}}    
	\def\HH{{\mathbb H}}    
	\def\AA{{\mathbb P}}     
	\def\KK{{\mathbb K}}     
	
	\def\cA{{\mathcal A}}  \def\cG{{\mathcal G}} \def\cM{{\mathcal M}} \def\cS{{\mathcal S}} \def\cB{{\mathcal B}}  \def\cH{{\mathcal H}} \def\cN{{\mathcal N}} \def\cT{{\mathcal T}} \def\cC{{\mathcal C}}  \def\cI{{\mathcal I}} \def\cO{{\mathcal O}} \def\cU{{\mathcal U}} \def\cD{{\mathcal D}}  \def\cJ{{\mathcal J}} \def\cP{{\mathcal P}} \def\cV{{\mathcal V}} \def\cE{{\mathcal E}}  \def\cK{{\mathcal K}} \def\cQ{{\mathcal Q}} \def\cW{{\mathcal W}} \def\cF{{\mathcal F}}  \def\cL{{\mathcal L}} \def\cR{{\mathcal R}} \def\cX{{\mathcal X}} \def\cY{{\mathcal Y}}  \def\cZ{{\mathcal Z}}
	
	
	\def\mfA{{\mathfrak A}} \def\mfA{{\mathfrak P}} \def\mfS{{\mathfrak S}}\def\mfZ{{\mathfrak Z}} \def\mfM{{\mathfrak M}} \def\mfQ{{\mathfrak Q}} \def\mfE{{\mathfrak E}} \def\mfL{{\mathfrak L}} \def\mfW{{\mathfrak W}} \def\mfR{{\mathfrak R}} \def\mfK{{\mathfrak K}} \def\mfX{{\mathfrak X}} \def\mfT{{\mathfrak T}} \def\mfJ{{\mathfrak J}} \def\mfC{{\mathfrak C}} \def\mfY{{\mathfrak Y}} \def\mfH{{\mathfrak H}} \def\mfV{{\mathfrak V}}\def\mfU{{\mathfrak U}}\def\mfG{{\mathfrak G}} \def\mfB{{\mathfrak B}} \def\mfI{{\mathfrak I}} \def\mfF{{\mathfrak F}} \def\mfN{{\mathfrak N}} \def\mfO{{\mathfrak O}} \def\mfD{{\mathfrak D}} 
	
	\def\mfa{{\mathfrak a}} \def\mfp{{\mathfrak p}} \def\mfs{{\mathfrak s}}  \def\mfz{{\mathfrak z}} \def\mfm{{\mathfrak m}} \def\mfq{{\mathfrak q}}  \def\mfe{{\mathfrak e}} \def\mfl{{\mathfrak l}} \def\mfw{{\mathfrak w}} \def\mfr{{\mathfrak r}} \def\mfk{{\mathfrak k}} \def\mfx{{\mathfrak x}} \def\mft{{\mathfrak t}} \def\mfj{{\mathfrak j}} \def\mfc{{\mathfrak c}} \def\mfy{{\mathfrak y}} \def\mfh{{\mathfrak h}} \def\mfv{{\mathfrak v}} \def\mfu{{\mathfrak u}} \def\mfg{{\mathfrak g}} \def\mfb{{\mathfrak b}} \def\mfi{{\mathfrak i}} \def\mff{{\mathfrak f}} \def\mfn{{\mathfrak n}} \def\mfo{{\mathfrak o}} \def\mfd{{\mathfrak d}} 
	
	\begin{abstract}
		In this work, we consider relative character varieties for
		representations of the 3-punctured sphere group  in PU(2,1).
		We provide necessary and sufficient conditions on the peripheral conjugacy classes, for such a representation to admit a decomposition as products of special elliptic elements. We prove that the representations satisfying these conditions form fibers of the so-called momentum map for PU(2,1). We apply these results to representations of the even subgroup of triangle groups, and describe components of the associated character variety.
	\end{abstract}

	\section{Introduction}
	
	In this paper, we study the representations of the group $\Gamma=\langle a,b,c \ | \ abc=1\rangle$ into $\operatorname{PU}(2,1)$, the isometry group of the complex hyperbolic plane $\mathbb{H}_\CC^2$. The group $\Gamma$ is isomorphic to the free group on two generators $\mathbb{F}_2$. A classical source of examples of such representations is the family of triangle groups, that are the following quotients of $\Gamma$:
	
	\begin{Defi*}
		Let $p,q,r\in\NN\cup\{\infty\}$ such that $1/p+1/q+1/r<1$. We define the \textit{triangle group}:
		$$\Gamma_{p,q,r}=\langle i_1,i_2,i_3 \ | \ i_k^2=(i_1i_2)^p=(i_2i_3)^q=(i_3i_1)^r=1\rangle,$$
		where the relation $(i_ki_l)^s=1$ is to omit if $s=\infty$, and its even subgroup
		$$\Gamma_{p,q,r}^{(2)}=\langle a,b,c \ | \ a^p=b^q=c^r=abc=1\rangle.$$
		We say that $\Gamma_{p,q,r}^{(2)}$ is a \textit{$(p,q,r)$-group}.
	\end{Defi*}
	
	The group $\Gamma_{p,q,r}$ acts naturally on $\mathbb{H}_\RR^2$ as the subgroup of $\operatorname{PO}(2,1)$ generated by the symmetries in the sides of a geodesic triangle with angles $\pi/p,\pi/q,\pi/r$. A family of representations in $\operatorname{PU}(2,1)$ is constructed in the following way: we first embed $\operatorname{PO}(2,1)$ into $\operatorname{PU}(2,1)$, and then deform the representation. This amounts to considering in $\mathbb{H}_\CC^2$ a triangle of complex lines with angles $\pi/p,\pi/q,\pi/r$ and sending each generator $i_k$ to a complex reflection $I_k$ of order two in each of the complex lines (i.e.\ a holomorphic isometry conjugate to the map $(z_1,z_2)\mapsto(z_1,-z_2)$; see Definition \ref{comprefl} in Section \ref{compref}). It is a classical fact that the space of such representations up to conjugation in $\operatorname{PU}(2,1)$ is compact and connected of dimension one. It is parameterized by the angular invariant $\alpha$ of the triangle (a projective invariant similar to a triple ratio, see \cite{Pra} and Section \ref{reflgrp} below). When $\alpha=\pi$, the three vertices of the triangle lie in a copy of $\mathbb{H}_\RR^2$, the representation restricts to $\operatorname{PO}(2,1)$ and coincides with the one introduced above.
	
	Schwartz stated a number of conjectures predicting the behaviour of the groups $\Gamma_{p,q,r}$ according to the values of $(p,q,r)$ and $\alpha$ (see \cite{Sch}). The first one essentially states that the discreteness of the representation only depends on the isometry type of two words in the group:
	
	\begin{Conj*}
		\emph{\cite{Sch}} We assume $3\leq p\leq q\leq r$. We consider a representation of $\Gamma_{p,q,r}$ as described above. Let $W_A=I_3I_2I_1I_2$ and $W_B=I_1I_2I_3$. The representation is discrete and faithful if and only if neither $W_A$ nor $W_B$ is elliptic.
	\end{Conj*}
	
	A certain number of partial results exist concerning this conjecture. It was proved for some particular values of $p,q,r$: see for instance \cite{GolPar} and \cite{Sch2} that prove the case $(\infty,\infty,\infty)$, or \cite{ParWil} for the case $(3,3,\infty)$; see also \cite{ParWanXie} for the case $(3,3,n)$. However, no general result exists today.
	
	The group $\Gamma_{p,q,r}^{(2)}$ is an index two subgroup of $\Gamma_{p,q,r}$, as well as the orbifold fundamental group of a $(\pi/p,\pi/q,\pi/r)$-triangular pillowcase. Understanding the representations of $\Gamma_{p,q,r}^{(2)}$ into $\operatorname{PU}(2,1)$ corresponds to describing the triples $(A,B,C)\in\operatorname{PU}(2,1)^3$ such that $A^p=B^q=C^r=ABC=1$. From now on, we always assume that the elements $A,B,C$ are elliptic.
	
	The representations of $\Gamma_{p,q,r}^{(2)}$ arising from representations of $\Gamma_{p,q,r}$ have the special property that there exists a triple of involutions $I_1,I_2,I_3\in\operatorname{PU}(2,1)$ such that $A=I_1I_2,B=I_2I_3,C=I_3I_1$. Such triples $(A,B,C)$ are sometimes called \textit{$\CC$-decomposable}, see Section 8.2 in \cite{PauWil} and Definition 1 in \cite{Wil}. Our main goal is to find other types of representations of $\Gamma_{p,q,r}^{(2)}$. To do so, we widen the notion of decomposability in the following way, allowing so-called \textit{complex reflections} instead of involutions. Complex reflections are elliptic isometries that fix pointwise a complex line, called the \textit{mirror}, and rotate around it. In particular they may have infinite order (see Section \ref{compref} for details).
	
	\begin{Defi*}
		We say that the triple $(A,B,C)$ is \textit{decomposable} if there exist three complex reflections $R_1,R_2,R_3$ such that $A=R_1R_2^{-1},B=R_2R_3^{-1},$ $C=R_3R_1^{-1}$.
	\end{Defi*}

In this paper, we address the following questions:
\begin{enumerate}[1.]
	\item Provide conditions on the conjugacy classes of $A,B,C$ guaranteeing the decomposability of the triple $(A,B,C)$;
	\item When these conditions are satisfied, describe the corresponding set of conjugacy classes of such triples up to global conjugation.
\end{enumerate}

We answer Question 1 in Section \ref{decell}, by providing a necessary and sufficient condition for an irreducible triple to be decomposable. Namely, we prove (Theorem \ref{caractdecomp}) that:

\begin{Theo*}
	An irreducible triple $(A,B,C)$ is decomposable if and only if there exist three eigenvalues $\lambda_A,\lambda_B,\lambda_C$ of any lifts in $\operatorname{SU}(2,1)$ such that $(\lambda_A\lambda_B\lambda_C)^3=1$.
\end{Theo*}

The examples obtained by taking the even subgroup of a representation of $\Gamma_{p,q,r}$ satisfy this condition, with $\lambda_A=\lambda_B=\lambda_C=1$. If the mirrors of the complex reflections intersect pairwise inside $\mathbb{H}_{\CC}^2$, they form a complex hyperbolic triangle associated to the triple $(A,B,C)$. We then make a link with the $\operatorname{SU}(2,1)$-character variety of the free group of rank two, denoted $\chi_{\operatorname{SU}(2,1)}(\mathbb{F}_2)$ (we recall the precise definitions in Section \ref{charvarsection}). The above theorem implies that a relative component of $\chi_{\operatorname{SU}(2,1)}(\mathbb{F}_2)$ that contains a decomposable triple contains only decomposable triples. We provide a parameterization of such components by geometric parameters. More precisely, we show that they are parameterized by the angle $\theta_1$ of the complex reflection $R_1$ and the angular invariant $\alpha$ of the complex hyperbolic triangle.

We then reinterpret these results in terms of the \textit{momentum map}.

\begin{Defi*}
	Let $\cC_1,\cC_2$ be fixed conjugacy classes in $\operatorname{PU}(2,1)$, and $\cG$ be the set of conjugacy classes in $\operatorname{PU}(2,1)$. Denote by $[A]$ the conjugacy class of an element $A$. The \textit{momentum map} is defined as
	$$\begin{array}{ccccc}
	\mu & : & \cC_1\times \cC_2 & \to & \cG \\
	& & (A,B) & \mapsto & [AB]. \\
	\end{array}$$
\end{Defi*}

In \cite{Pau}, Paupert proposes a method to describe the image of $\mu$ when $\cC_1,\cC_2$ are elliptic conjugacy classes (see \cite{FalWen} for the case where the classes are loxodromic). His starting point is to identify the image of the subset of $\cC_1\times \cC_2$ formed by \textit{reducible pairs} $(A,B)$ (i.e.\ $A$ and $B$ have a common eigenspace). This subspace is one-dimensional, and known as the \textit{reducible skeleton}.

Now, Question 2 amounts to describing the fiber of the momentum map above a conjugacy class $\cC_3$ of order $r$, when $\cC_1,\cC_2$ have order $p,q$ and the three conjugacy classes satisfy the decomposability condition. Our numerical condition for the decomposability of a triple may be translated in this setting as the following result (Theorem \ref{fiber}):

\begin{Theo*}
	An irreducible triple $(A,B,(AB)^{-1})$ is decomposable if and only if the conjugacy class of $(AB)^{-1}$ lies on the reducible skeleton. As a consequence, the whole fiber of $\mu$ above the conjugacy class of $(AB)^{-1}$ is decomposable.
\end{Theo*}

Finally, we apply our results to representations of $(p,q,r)$-groups in Section \ref{trigroups}. Our work enables us to parameterize geometrically the $\RR$-Fuchsian component of $\chi_{\operatorname{SU}(2,1)}(\Gamma_{p,q,r}^{(2)})$ (i.e.\ the deformation component of the embedding $\Gamma_{p,q,r}^{(2)}\to\operatorname{PO}(2,1)\subset\operatorname{PU}(2,1)$). The results of this paper have the following consequence (Theorem \ref{RFuchs}):

\begin{Theo*}
	The $\RR$-Fuchsian component of $\chi_{\operatorname{SU}(2,1)}(\Gamma_{p,q,r}^{(2)})$ is entirely composed of decomposable representations. Besides, it contains a unique reducible point and is topologically a sphere. In particular, it is compact.
\end{Theo*}

We construct explicitly this parameterization in the case of the $(3,3,4)$-group. Similar computations appear in an unpublished note by Thistlethwaite \cite{Thi}. In this paper, we show how this method also gives way to parameterizing non-$\RR$-Fuchsian components of $\chi_{\operatorname{SU}(2,1)}(\Gamma_{p,q,r}^{(2)})$. We give an example with the $(4,5,20)$-group.

The results of this paper are part of a PhD thesis under the supervision of Antonin Guilloux and Pierre Will. We thank Arnaud Maret, Elisha Falbel and Julien Paupert for the useful discussions. This work is supported by the French National Research Agency in the framework of the « France 2030 » program  (ANR-15-IDEX-0002) and by the LabEx PERSYVAL-Lab (ANR-11-LABX-0025-01).
	
	\section{The complex hyperbolic space of dimension 2}\label{comphyp}
	
	\subsection{Generalities}
	
	We start by reviewing classical notions about the complex hyperbolic space of dimension two $\mathbb{H}_\CC^2$, highlighting the objects we will be studying next. General references are \cite{Gol} or \cite{Par}.
	
	Let us consider the following Hermitian form of signature $(2,1)$, defined for any $z,w$ in $\CC^3$ by
	$\langle z,w\rangle=z_1\overline{w_1}+z_2\overline{w_2}-z_3\overline{w_3}$. Its matrix in the canonical basis is
	$$J=\begin{pmatrix}
	1 & 0 & 0 \\
	0 & 1 & 0 \\
	0 & 0 & -1 
	\end{pmatrix}.$$
	
	We will use the ball model for $\mathbb{H}_\CC^2$, which is defined as follows:
	
	\begin{Defi}
		Let $\mathbb{H}_\CC^2$ be the projectivization of the negative cone $V^-=\{z\in\CC^3|\langle z,z\rangle<0\}$. Its boundary $\partial\mathbb{H}_\CC^2$ is the projectivization of the null cone $V^0=\{z\in\CC^3|\langle z,z\rangle=0\}$. We denote by $\overline{\mathbb{H}_\CC^2}$ the union $\mathbb{H}_\CC^2\cup\partial\mathbb{H}_\CC^2$.
	\end{Defi}

\begin{Rema}
	We say that an element of $V^-$ (respectively $V^+$, respectively $V^0$) is of \textit{negative} (respectively \textit{positive}, respectively \textit{null}) \textit{type}.
\end{Rema}
	
	We will usually work in the affine chart $\{z_3=1\}$. This chart identifies $\mathbb{H}_\CC^2$ with the unit ball of $\CC^2$. For an element $z=(z_1,z_2)$ in $\mathbb{H}_\CC^2$, we shall denote by $\tilde{z}=(z_1,z_2,1)\in\CC^3$ its lift to this affine chart. The following formula defines a metric $d$ on $\mathbb{H}_\CC^2$ (see Section 3.1.6 in \cite{Gol}):
	\begin{equation*}
		\forall z,w \in \mathbb{H}^2_\CC, \  \cosh^2\left(\frac{d(z,w)}{2}\right)=\frac{\langle\tilde{z},\tilde{w}\rangle\langle\tilde{w},\tilde{z}\rangle}{\langle\tilde{z},\tilde{z}\rangle\langle\tilde{w},\tilde{w}\rangle}.
	\end{equation*}
	
	Let us recall that $\operatorname{U}(2,1)$ is the subgroup of $\operatorname{GL}_3(\CC)$ that preserves the Hermitian form $\langle\cdot,\cdot\rangle$. We denote by $\operatorname{PU}(2,1)$ the quotient of this group by its center (which is the set of homotheties). The group $\operatorname{PU}(2,1)$ acts naturally on $\mathbb{P}(\CC^3)$, preserving $V^-$ and $V^0$. Therefore, it acts on $\mathbb{H}_\CC^2$. The following classical theorem describes the main properties of this action.
	
	\begin{Theo} \label{isomclassif}
		\emph{\cite{Par}} $\operatorname{PU}(2,1)$ is the group of holomorphic isometries of $\mathbb{H}_\CC^2$, and acts transitively. Moreover, the stabilizer of a point in $\mathbb{H}_\CC^2$ under the action of $\operatorname{PU}(2,1)$ is $\mathbb{P}(\operatorname{U}(2)\times\operatorname{U}(1))$.
	\end{Theo}
	
	Just like in real hyperbolic geometry, any isometry belongs to one of the three (mutually exclusive) following classes: it is \textit{loxodromic} if it fixes exactly two points in $\partial\mathbb{H}_\CC^2$ and none in $\mathbb{H}_\CC^2$, \textit{parabolic} if it fixes exactly one point in $\partial\mathbb{H}_\CC^2$ and none in $\mathbb{H}_\CC^2$, and \textit{elliptic} if it fixes at least one point in $\mathbb{H}_\CC^2$.

One can refine the elliptic case by looking at the possibility of a repeated eigenvalue.
If the representatives in $\operatorname{U}(2,1)$ of an isometry are diagonalisable with eigenvalues that are pairwise distinct, we say that this element is \textit{regular}; when two eigenvalues are equal we call it \textit{special}.
	
	\begin{Rema}\label{ellip}
		Any elliptic element is diagonalisable, and admits a lift in $\operatorname{U}(2,1)$ conjugate to a matrix of the form 
		$$E(\alpha_1,\alpha_2)=\begin{pmatrix}
		e^{i\alpha_1} & 0 & 0 \\
		0 & e^{i\alpha_2} & 0 \\
		0 & 0 & 1 
		\end{pmatrix},$$
		with $0\leq\alpha_2\leq\alpha_1<2\pi$. This element is \emph{regular elliptic} when $0<\alpha_2<\alpha_1<2\pi$. Otherwise it is \emph{special elliptic}. Besides, it preserves a copy of $\mathbb{H}_\RR^2$ if and only if $(\alpha_1,\alpha_2)=(\alpha,2\pi-\alpha)$ for some $\alpha$.
	\end{Rema}

	\begin{Defi}\label{anglepair}
		The pair $(\alpha_1,\alpha_2)$ is called the \textit{angle pair} of the element $E(\alpha_1,\alpha_2)$.
	\end{Defi}
	
	The conjugacy class of an elliptic element in $\operatorname{PU}(2,1)$ is uniquely determined by its angle pair. Another useful tool is the box product:
	
	\begin{Defi}
		For $u,v\in\CC^{3}$, their \textit{box product} is
		$$u\boxtimes v=J\cdot\overline{u\wedge v},$$
		where $J$ is the matrix of the Hermitian form in the canonical basis and $\wedge$ is the usual exterior product in $\CC^{3}$.
	\end{Defi}
	
	Note that $u\boxtimes v$ is orthogonal to $u$ and $v$ for the Hermitian form $\langle\cdot,\cdot\rangle$.
	
	\begin{Defi}\label{crossratio}
		We define the following cross-ratio for $a,b,c,d\in\mathbb{P}(\CC^3)$:
		$$X(a,b,c,d)=\frac{\langle\tilde{c},\tilde{a}\rangle\langle\tilde{d},\tilde{b}\rangle}{\langle\tilde{c},\tilde{b}\rangle\langle\tilde{d},\tilde{a}\rangle}.$$
	\end{Defi}
	
The cross-ratio $X(a,b,c,d)$ can be used to decide whether the complex lines $(ab)$ and $(cd)$ are orthogonal (see Definition \ref{comprefl} below).
	
	\begin{Prop}\label{crossratio=1}
		Assume that $a,b,c,d\in\mathbb{P}(\CC^3)$ are pairwise non orthogonal to each other. The complex lines $(ab)$ and $(cd)$ are orthogonal if and only if we have $X(a,b,c,d)=1$.
	\end{Prop}

	\begin{proof}
		We have:
		$$\langle\tilde{a}\boxtimes\tilde{b},\tilde{c}\boxtimes\tilde{d}\rangle=\langle\tilde{d},\tilde{a}\rangle\langle\tilde{c},\tilde{b}\rangle-\langle\tilde{c},\tilde{a}\rangle\langle\tilde{d},\tilde{b}\rangle.$$
		Now, the complex lines $(ab)$ and $(cd)$ are orthogonal if an only if $\langle\tilde{a}\boxtimes\tilde{b},\tilde{c}\boxtimes\tilde{d}\rangle=0$, which is equivalent to $X(a,b,c,d)=1$.
	\end{proof}

The following classical result classifies triples of points in $\mathbb{P}(\CC^3)$ modulo the diagonal action of $\operatorname{PU}(2,1)$.

\begin{Prop}\label{reflexistence}
	Let $(p_1,p_2,p_3),(q_1,q_2,q_3)$ be triples in $\mathbb{P}(\CC^3)^3$. There exists an element of $\operatorname{PU}(2,1)$ sending $(p_1,p_2,p_3)$ to $(q_1,q_2,q_3)$ if and only if there exist lifts $(\tilde{p_1},\tilde{p_2},\tilde{p_3}),(\tilde{q_1},\tilde{q_2},\tilde{q_3})\in(\CC^3)^3$ such that
	\begin{itemize}
		\item $\forall i,\langle\tilde{p_i},\tilde{p_i}\rangle=\langle\tilde{q_i},\tilde{q_i}\rangle$;
		\item $\forall i,j,|\langle\tilde{p_i},\tilde{p_j}\rangle|=|\langle\tilde{q_i},\tilde{q_j}\rangle|$;
		\item $\operatorname{arg}(\langle\tilde{p_1},\tilde{p_2}\rangle\langle\tilde{p_2},\tilde{p_3}\rangle\langle\tilde{p_3},\tilde{p_1}\rangle)=\operatorname{arg}(\langle\tilde{q_1},\tilde{q_2}\rangle\langle\tilde{q_2},\tilde{q_3}\rangle\langle\tilde{q_3},\tilde{q_1}\rangle)$.
	\end{itemize}
\end{Prop}
	
	\subsection{Complex reflections}\label{compref}
	
	Let us develop some notions and facts about complex lines of $\mathbb{H}_\CC^2$ that we will use later on. For more details, see Section 5.2 in \cite{Par}.
	
	\begin{Defi}\label{comprefl}
		We call \textit{complex line} the intersection of $\mathbb{H}_\CC^2$ with $\mathbb{P}(W)$, where $W$ is a complex plane in $\CC^3$ that intersects $V^-$. The  one-dimensional subspace $W^\perp\subset\CC^3$ is known as the line of \textit{polar vectors} to the complex line. These vectors are of positive type. We say that a polar vector $c$ is \textit{normalized} if $\langle c,c\rangle=1$. 
	\end{Defi}

\begin{Rema}
	A projective line in $\mathbb{P}(\CC^3)$ may be disjoint from $\mathbb{H}_\CC^2$. It also has a line of polar vectors, which are of negative type.
\end{Rema}

\begin{Exem}
	The standard complex line through the origin is
	$$C_0:=\{(z,0)\in\CC^2| \ |z|^2<1\}=\mathbb{P}(\{z_1,0,z_3\})\cap\mathbb{H}_\CC^2.$$
	$C_0$ is polar to the normalized vector $c_0=(0,1,0)$. The complex line $C_0$ is the image of $\mathbb{H}_\CC^1$ by an isometric embedding into $\mathbb{H}_\CC^2$.
\end{Exem}
	
	\begin{Prop}
		\emph{\cite{Par}} $\operatorname{PU}(2,1)$ acts transitively on the set of complex lines. The stabilizer of a complex line under the action of $\operatorname{PU}(2,1)$ is conjugate to the stabilizer of $C_0$ in $\operatorname{PU}(2,1)$, which is $\mathbb{P}(\operatorname{U}(1,1)\times\operatorname{U}(1))$ (compare with Theorem \ref{isomclassif}).
	\end{Prop}
	
	\begin{Rema}
		Any two points in $\mathbb{H}_\CC^2$ are contained in a unique complex line.
	\end{Rema}
	
	A complex line is supported by a projective line in the projective plane. We therefore have three possibilities for the relative position of two distinct complex lines:
	\begin{Defi}
		\begin{itemize}
			\item If they intersect in a point in $\mathbb{H}_\CC^2$, we say that they are \textit{intersecting} ;
			\item If their supporting projective lines intersect in a point outside $\overline{\mathbb{H}_\CC^2}$, and therefore are disjoint in $\overline{\mathbb{H}_\CC^2}$, we call them \textit{ultraparallel} ;
			\item If their closures intersect in a point in $\partial\mathbb{H}_\CC^2$, we call them \textit{asymptotic}.
		\end{itemize}
	\end{Defi}
	
	We describe pairs of complex lines in the following way:
	
	\begin{Defi}\label{angledist}
		Let $C_1,C_2$ be two intersecting complex lines, with normalized polar vectors $c_1,c_2$. The \textit{complex angle} $\phi(C_1,C_2)\in[0,\pi/2]$ between $C_1$ and $C_2$ is defined as
		$$\cos(\phi(C_1,C_2))=|\langle c_1,c_2\rangle|.$$
	\end{Defi}
	
	If $C_1,C_2$ are two ultraparallel complex lines, with normalized polar vectors $c_1,c_2$, then the distance $l$ between $C_1$ and $C_2$ is defined as
	$$\cosh(l/2)=|\langle c_1,c_2\rangle|.$$
	
	\begin{Rema}
		We notice that a regular elliptic element $A$ has three fixed points in $\mathbb{P}(\CC^3)$: one inside $\mathbb{H}_\CC^2$, and two outside. These last two are vectors polar to the two complex lines that are globally preserved by $A$, and on which it acts by rotations (through angles $\alpha_1$ and $\alpha_2$ in Remark \ref{ellip}). These two complex lines intersect in the fixed point that is inside $\mathbb{H}_\CC^2$.
	\end{Rema}
	
	We now divide the special elliptic isometries into two types:
	
	\begin{Defi}
		\begin{itemize}
			\item Let $C$ be a complex line of $\mathbb{H}_\CC^2$. A \textit{complex reflection in the complex line $C$} is an isometry whose fixed point set is exactly $C$. The line $C$ is called the \textit{mirror} of the complex reflection.
			\item Let $c$ be a point in $\mathbb{H}_\CC^2$. A \textit{complex reflection in $c\in\mathbb{H}_\CC^2$} is a special elliptic isometry whose fixed point set is $c$. It also fixes pointwise the projective line polar to $c$, which we sometimes call the mirror as well.
		\end{itemize}
	\end{Defi}

In both cases, the complex reflection has a unique projective line of fixed points, and a unique isolated fixed point polar to this line.
	
	A complex reflection in the complex line $C_0$ defined above is represented by the matrix 
	$$R(\eta)=\begin{pmatrix}
	1 & 0 & 0 \\
	0 & \eta & 0 \\
	0 & 0 & 1 
	\end{pmatrix},$$
	where $\eta=e^{i\theta}\neq1$ is a complex number of modulus one, called the \textit{rotation factor} of the reflection. Any complex reflection in a complex line is conjugate to $R(\eta)$ for some $\eta$. The action of $R(\eta)$ on $\mathbb{H}_\CC^2$ is given by $(z_1,z_2)\mapsto(z_1,\eta z_2)$: it fixes the first coordinate axis and acts by a rotation of angle $\theta$ on the second axis. A representative in $\operatorname{SU}(2,1)$ depends on the choice of a cube root $\eta^{1/3}$ and is given by the matrix:
	$$\tilde{R}(\eta)=\begin{pmatrix}
	\eta^{-1/3} & 0 & 0 \\
	0 & \eta^{2/3} & 0 \\
	0 & 0 & \eta^{-1/3} 
	\end{pmatrix}.$$
	
Similarly, a complex reflection in the origin $(0,0,1)$ is represented by 
$$S(\eta)=\begin{pmatrix}
\eta & 0 & 0 \\
0 & \eta & 0 \\
0 & 0 & 1 
\end{pmatrix}.$$
Its action on $\mathbb{H}_\CC^2$ is given by $(z_1,z_2)\mapsto(\eta z_1,\eta z_2)$: it fixes the origin and acts by a rotation of angle $\theta$ on each complex line through this point.

\begin{Lemm}\label{reflectionformula}
	\emph{\cite{Pra}} A complex reflection $R$ of rotation factor $\eta$ whose isolated fixed point is $c\in\mathbb{P}(\CC^3)$ is the projectivization of the map $\tilde{R}$ defined over $\CC^3$ by:
	\begin{equation*}
	\forall z\in\mathbb{H}_\CC^2, \ \tilde{R}(\tilde{z})=\tilde{z}+(\eta-1)\frac{\langle \tilde{z},\tilde{c}\rangle}{\langle \tilde{c},\tilde{c}\rangle}\tilde{c}.
	\end{equation*}
\end{Lemm}

	\begin{Lemm}\label{fppairs}
		Let $A\in\operatorname{PU}(2,1)$ be a non-trivial isometry fixing two distinct points $P_1,P_2\in\mathbb{P}(\CC^3)\backslash\partial\mathbb{H}^2_\CC$. We assume their lifts to $\CC^3$ to be non orthogonal to each other.
		\begin{itemize}
			\item If the projective line $(P_1P_2)$ intersects $\mathbb{H}^2_\CC$, then $A$ is a complex reflection in $(P_1P_2)$;
			\item If the projective line $(P_1P_2)$ does not intersect $\mathbb{H}^2_\CC$, then $A$ is a complex reflection in the point polar to $(P_1P_2)$.
		\end{itemize}
		 
	\end{Lemm}
	
	\begin{proof}
		A quick inspection of the situation shows that $A$ cannot be loxodromic or parabolic. If $A$ is regular elliptic, then its three fixed points are pairwise orthogonal, which is excluded by hypothesis. Therefore $A$ must be special elliptic. This concludes the proof, using the definitions above.
	\end{proof}

\subsection{A normalization for a triple of complex reflections}\label{reflnorm}

We introduce here a convenient way of writing irreducible triples of complex reflections (i.e.\ their eigenspaces are pairwise disjoint). It will be useful later on for parameterizations.

Let $R_1,R_2,R_3$ be three complex reflections in the pairwise distinct complex lines $C_1,C_2,C_3$, with normalized polar vectors $c_1,c_2,c_3$ that form a base of $\CC^3$, and rotation factors $\eta_1,\eta_2,\eta_3$. For $k\in\{1,2,3\}$, denote $r_k=|\langle c_{k-1},c_{k+1}\rangle|$. We recover Definition \ref{angledist}:
\begin{itemize}
	\item if $r_k<1$, $C_{k-1}$ and $C_{k+1}$ are intersecting, and $r_k$ is equal to the cosine of the angle between them;
	\item if $r_k>1$, $C_{k-1}$ and $C_{k+1}$ are ultraparallel and $r_k$ is equal to the hyperbolic cosine of the half-distance between them.
\end{itemize}

Define
$$\alpha=\arg\left(\prod_{k=1}^{3}\langle c_{k-1},c_{k+1}\rangle\right).$$
Up to multiplying each $c_k$ by a complex number of modulus $1$, we may assume that for each $k$, $\langle c_{k-1},c_{k+1}\rangle=r_ke^{i\alpha/3}$. Denote $u=e^{i\alpha/3}$.

The matrix of the Hermitian form in this basis is
$$H=\begin{pmatrix}
1 & r_3u & r_2u^{-1} \\
r_3u^{-1} & 1 & r_1u \\
r_2u & r_1u^{-1} & 1 
\end{pmatrix}.$$
		
The data $(r_1,r_2,r_3,u)$ classifies triples of complex lines up to the action of $\operatorname{PU}(2,1)$. See for instance \cite{Pra} where another normalization is used, and Section \ref{reflgrp}.

\begin{Lemm}\label{reflnormlemma}
	The lifts of $R_1,R_2,R_3$ in the basis $(c_1,c_2,c_3)$ are
	$$\tilde{R_1}=\begin{pmatrix}
	\eta_1 & r_3(\eta_1-1)u^{-1} & r_2(\eta_1-1)u \\
	0 & 1 & 0 \\
	0 & 0 & 1
	\end{pmatrix},$$
	$$\tilde{R_2}=\begin{pmatrix}
	1 & 0 & 0 \\
	r_3(\eta_2-1)u & \eta_2 & r_1(\eta_2-1)u^{-1} \\
	0 & 0 & 1
	\end{pmatrix},$$
	$$\tilde{R_3}=\begin{pmatrix}
	1 & 0 & 0 \\
	0 & 1 & 0 \\
	r_2(\eta_3-1)u^{-1} & r_1(\eta_3-1)u & \eta_3
	\end{pmatrix}.$$
\end{Lemm}

\begin{proof}
	The proof follows directly from applying the formula of Lemma \ref{reflectionformula} to $c_1,c_2,c_3$.
\end{proof}

\section{Decomposing regular elliptics}\label{decell}

We now describe the notion of decomposability in products of complex reflections.

\subsection{Definition and characterization}

\begin{Defi}\label{decomp}
	\begin{itemize}
		\item Let $A,B,C\in\operatorname{PU}(2,1)$ be regular elliptic elements with $ABC=1$. We say that $(A,B,C)$ is a \textit{decomposable triple} if there exist three complex reflections $R_1,R_2,R_3$ such that $A=R_1R_2^{-1},B=R_2R_3^{-1},C=R_3R_1^{-1}$.
		\item If $\rho$ is a representation of $\Gamma$ in $\operatorname{PU}(2,1)$, and $A=\rho(a),B=\rho(b),C=\rho(c)$ are regular elliptic, we say that $\rho$ is \textit{decomposable} if $(A,B,C)$ is a decomposable triple.
	\end{itemize}
\end{Defi}

\begin{Rema}
	The decomposability property is invariant up to global conjugation of the triple.
\end{Rema}

\begin{Defi}\label{irred}
	Let $A,B,C\in\operatorname{PU}(2,1)$ be regular elliptic elements with $ABC=1$. We say that $(A,B,C)$ is \textit{irreducible} when the eigenspaces of $A,B,C$ are pairwise distinct; otherwise, $(A,B,C)$ is \textit{reducible}.
\end{Defi}

We will describe a necessary and sufficient condition for an irreducible representation to be decomposable. This condition will be expressed with a choice of preferred eigenvectors of the three regular elliptic elements. We describe this choice in the following definition.

\begin{Defi}
	Let $\rho$ be a representation of $\Gamma$ in $\operatorname{PU}(2,1)$, with $A=\rho(a),B=\rho(b),C=\rho(c)$ regular elliptic elements.
	\begin{enumerate}[i)]
		\item Let $P_A,P_B,P_C\in\mathbb{P}(\CC^3)\backslash\partial\mathbb{H}_\CC^2$  be three fixed points. We say that they are in \textit{irreducible configuration} when they are pairwise distinct and non orthogonal.
		\item Suppose we choose fixed points in irreducible configuration. Let $\tilde{A},\tilde{B},\tilde{C}$ be representatives of $A,B,C$ in $\operatorname{SU}(2,1)$ (this means that $\tilde{A}\tilde{B}\tilde{C}=\omega I$, where $\omega$ is a cube root of unity) and $\tilde{P_A},\tilde{P_B},\tilde{P_C}$ be lifts of $P_A,P_B,P_C$ in $\CC^3$. Consequently, $\tilde{A}\tilde{P_A}=\lambda_A\tilde{P_A},\tilde{B}\tilde{P_B}=\lambda_B\tilde{P_B},\tilde{AB}\tilde{P_C}=\lambda_{AB}\tilde{P_C}=\omega\overline{\lambda_C}\tilde{P_C}$, where $\lambda_A,\lambda_B,\lambda_{AB}$ are unit modulus eigenvalues. We call them \textit{preferred eigenvalues}.
	\end{enumerate}
\end{Defi}

\begin{Rema}
	\begin{itemize}
		\item Each fixed point of $A,B,C$ has negative or positive type. So an irreducible configuration $(P_A,P_B,P_C)$ has a type $(\epsilon_A,\epsilon_B,\epsilon_C)\in\{+,-\}^3$. See Section \ref{sphhyp} for a geometric interpretation.
		\item If the triple $(A,B,C)$ is irreducible (in the sense of Definition \ref{irred}), then for any triple $(\epsilon_A,\epsilon_B,\epsilon_C)\in\{+,-\}^3$, there exists a triple $(P_A,P_B,P_C)$ of type $(\epsilon_A,\epsilon_B,\epsilon_C)$ which is in irreducible configuration.
	\end{itemize}
\end{Rema}

We now state our first result.

\begin{Theo}\label{caractdecomp}
	Let $\rho$ be a representation of $\Gamma$ in $\operatorname{PU}(2,1)$, and let $P_A,P_B,P_C$ be fixed points of $(A,B,C)$ in irreducible configuration. Let $\lambda_A,\lambda_B,\lambda_C$ be preferred eigenvalues. The following are equivalent:
	\begin{enumerate}[i)]
		\item The triple $(A,B,C)$ is decomposable, and the corresponding reflections $R_1,R_2,R_3$ have respective mirrors $(P_AP_C),(P_AP_B),(P_BP_C)$;
		\item Let $P_{BA}:=A^{-1}P_C$. The projective lines $(P_AP_B)$ and $(P_CP_{BA})$ are orthogonal;
		\item $(\lambda_A\lambda_B\lambda_C)^3=1$.
	\end{enumerate}
\end{Theo}

\begin{Rema}
	\begin{itemize}
		\item The fact that $P_{BA}:=A^{-1}P_C$ indeed defines a fixed point of $BA$ follows from $BAP_{BA}=A^{-1}(AB)AP_{BA}=A^{-1}(AB)AA^{-1}P_C=A^{-1}P_C=P_{BA}.$
		\item The product $(\lambda_A\lambda_B\lambda_C)^3$ is independent of the choice of lifts to $\operatorname{SU}(2,1)$.
	\end{itemize}
\end{Rema}

\begin{center}
	\begin{tikzpicture}
	\draw (0,2) node[above]{$P_A$} ;
	\draw (0,-2) node[below]{$P_B$} ;
	\draw (2,0) node[right]{$P_C$} ;
	\draw (-2,0) node[left]{$P_{BA}$} ;
	\draw (1/4,13/8) node[below]{$A$} ;
	\draw (-1/4,-13/8) node[above]{$B$} ;
	\draw (2,0) -- (0,2) -- (-2,0) -- (0,-2) -- (2,0);
	\draw [dashed] (-2,0) -- (2,0);
	\draw (-1/4,3/4+1) arc (225:315:1/3) ;
	\draw (1/4,-3/4-1) arc (45:135:1/3) ;
	\draw [dashed] (0,2) -- (0,-2);
	\draw (1/4,0) -- (1/4,1/4) -- (0,1/4);
	\draw (1/4,3/4+1) -- (1/4-0.15,3/4+1-0.15);
	\draw (1/4-0.18,3/4+1+0.025) -- (1/4,3/4+1);
	\draw (-1/4,-3/4-1) -- (-1/4+0.15,-3/4-1+0.15);
	\draw (-1/4+0.18,-3/4-1-0.025) -- (-1/4,-3/4-1);
	\end{tikzpicture}
\end{center}

\begin{proof}
	$i)\Rightarrow ii)$: Since $AP_{BA}=P_C$ and $A=R_1R_2^{-1}$, we get $R_2^{-1}P_{BA}=R_1^{-1}P_C=P_C$. Consequently, $R_2$ is a complex reflection that fixes the complex line $(P_AP_B)$ pointwise, and globally preserves the line $(P_CP_{BA})$: this implies that they are orthogonal to each other.
	
	$ii)\Leftrightarrow iii)$: We compute $X(P_A,P_B,P_C,P_{BA})$ (see Definition \ref{crossratio}). Since $AP_{BA}=P_C$ and $BP_C=P_{BA}$, there exist complex numbers $\mu,\nu$ such that $\tilde{A}\tilde{P_{BA}}=\mu\tilde{P_C}$ and $\tilde{B}\tilde{P_C}=\nu\tilde{P_{BA}}$. But $\tilde{AB}\tilde{P_C}=\omega\overline{\lambda_C}\tilde{P_C}$. So $\omega\overline{\lambda_C}=\mu\nu$. We have
	$$\langle\tilde{P_C},\tilde{P_B}\rangle=\langle\tilde{B}\tilde{P_C},\tilde{B}\tilde{P_B}\rangle=\nu\overline{\lambda_{B}}\langle\tilde{P_{BA}},\tilde{P_B}\rangle.$$
	Similarly, $\langle\tilde{P_{BA}},\tilde{P_A}\rangle=\mu\overline{\lambda_{A}}\langle\tilde{P_C},\tilde{P_A}\rangle$.
	We now compute the Hermitian cross-ratio:
	$$X(P_A,P_B,P_C,P_{BA})=\frac{\langle\tilde{P_C},\tilde{P_A}\rangle\langle\tilde{P_{BA}},\tilde{P_B}\rangle}{\nu\overline{\lambda_{B}}\langle\tilde{P_{BA}},\tilde{P_B}\rangle\mu\overline{\lambda_{A}}\langle\tilde{P_C},\tilde{P_A}\rangle}=\frac{\lambda_{A}\lambda_{B}}{\omega\overline{\lambda_C}}.$$
	This number is equal to $1$ if and only if $\lambda_A\lambda_B\lambda_C=\omega$, which concludes the proof using Proposition \ref{crossratio=1}.
	
	$ii)\Rightarrow i)$: The computations of the preceeding step show that 
	$$1=X(P_A,P_B,P_C,P_{BA})=\frac{\langle\tilde{P_A},\tilde{P_B}\rangle\langle\tilde{P_{B}},\tilde{P_C}\rangle\langle\tilde{P_C},\tilde{P_A}\rangle}{\langle\tilde{P_A},\tilde{P_B}\rangle\langle\tilde{P_{B}},\tilde{P_{BA}}\rangle\langle\tilde{P_{BA}},\tilde{P_A}\rangle}\cdot\frac{|\langle\tilde{P_{B}},\tilde{P_{BA}}\rangle|^2}{|\langle\tilde{P_{B}},\tilde{P_C}\rangle|^2}$$ 
	By Proposition \ref{reflexistence}, there exists a unique isometry sending $(P_A,P_B,P_C)$ to $(P_A,P_B,P_{BA})$. In particular, it fixes $P_A$ and $P_B$. Using Lemma \ref{fppairs}, we obtain that this isometry is a complex reflection $R_2$ in $(P_AP_B)$ such that $R_2(P_C)=P_{BA}$. In particular, $AR_2$ fixes $P_A$. Then, since $AP_{BA}=P_C$, $AR_2$ fixes $P_C$. Therefore according to Lemma \ref{fppairs}, $AR_2$ is a complex reflection in $(P_AP_C)$. We call it $R_1$. Since $AP_{BA}=P_C$, we have $BP_C=P_{BA}$, so $R_2^{-1}B$ fixes $P_B$ and $P_C$. By the same argument as before, it is a complex reflection in $(P_BP_C)$, which we denote by $R_3^{-1}$. By construction we have $A=R_1R_2^{-1}$ and $B=R_2R_3^{-1}$. Finally, $C=(AB)^{-1}=(R_1R_2^{-1}R_2R_3^{-1})^{-1}=R_3R_1^{-1}$, which concludes the proof of this implication.

\end{proof}

\subsection{Spherical and hyperbolic decomposable triples}\label{sphhyp}

Let us now describe geometrically what happens according to the positions of $P_A,P_B,P_C$.

\begin{enumerate}
	\item If $P_A,P_B,P_C$ are all inside $\mathbb{H}_\CC^2$ (i.e.\ the triple has type $(-,-,-)$), then we obtain a set of complex reflections whose mirrors intersect pairwise. We specify this situation by saying that $(A,B,C)$ is a \textit{spherical decomposable triple}.
	\item If $P_A,P_B,P_C$ are all outside $\mathbb{H}_\CC^2$ (i.e.\ the triple has type $(+,+,+)$), then:
	\begin{itemize}
		\item if all the mirrors intersect $\mathbb{H}_\CC^2$, we obtain a set of complex reflections whose mirrors are pairwise ultraparallel;
		\item if some mirrors do not intersect $\mathbb{H}_\CC^2$, we obtain a decomposition that includes complex reflections in points.
	\end{itemize}
	We specify this situation by saying that $(A,B,C)$ is a \textit{hyperbolic decomposable triple}.
	\item If one point among $P_A,P_B,P_C$ is inside $\mathbb{H}_\CC^2$ and two points are outside (i.e.\ the triple has type $(-,+,+)$ up to permutation), we obtain a set of reflections with two mirrors intersecting, and the third either ultraparallel to the other two or entirely outside $\mathbb{H}_\CC^2$.
	\item If one point among $P_A,P_B,P_C$ is outside $\mathbb{H}_\CC^2$ and two points are inside (i.e.\ the triple has type $(+,-,-)$ up to permutation), we obtain a set of complex reflections in lines with one mirror intersecting the two others, these last two being ultraparallel.
\end{enumerate}

The cases 3 and 4 are called \textit{mixed decomposable}, and we regroup cases 1 and 2 under the name of \textit{pure decomposable}.

We will focus on decompositions in products of complex reflections whose mirrors intersect pairwise, or are pairwise disjoint, i.e.\ pure decomposable triples.

\subsection{Traces and decompositions}\label{tracesanddecomp}

To construct explicit decompositions, we consider traces of representatives in $\operatorname{SU}(2,1)$. Similar computations appear in \cite{Fra} for other purposes.

\begin{Defi}
	\cite{Gol} We call \textit{deltoid} the subset of $\CC$ defined by: $$\Delta=\{2e^{i\theta}+e^{-2i\theta},\theta\in[-\pi,\pi]\}.$$
	We say that a point $z\in\CC$ is inside $\Delta$ if it is in the bounded component of $\CC\backslash\Delta$. See Figure \ref{deltoid}.
\end{Defi}

\begin{figure}
	\centering
	\scalebox{0.3}{\includegraphics{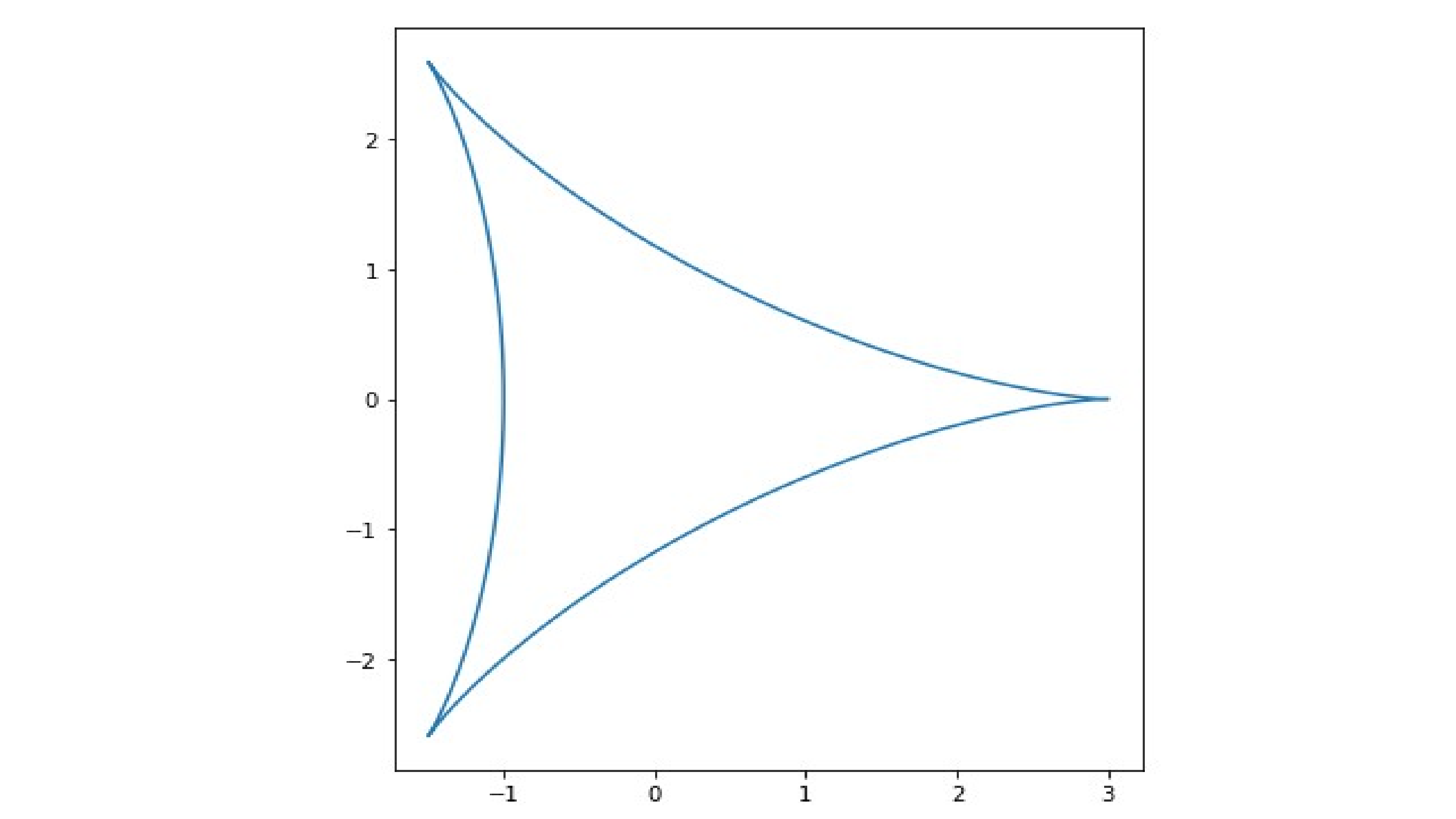}}
	\caption{The deltoid $\Delta$. \label{deltoid}}
\end{figure}

The following result is part of Theorem 6.2.4 in \cite{Gol}.

\begin{Prop}
	An element of $\operatorname{SU}(2,1)$ is regular elliptic if and only if its trace is inside $\Delta$, whereas the trace of a special elliptic element is on $\Delta$.
\end{Prop}

A computation of the trace of a product of two complex reflections leads to the following result, in the spirit of \cite{Pra}:

\begin{Lemm}\label{traceformula}
	Let $R_1,R_2$ be two complex reflections with rotation factors $\eta_1=e^{i\theta_1},\eta_2=e^{i\theta_2}$. Then:
	\begin{enumerate}[i)]
		\item If their mirrors intersect with angle $\phi$,
		\begin{equation}
		\begin{aligned}
		\operatorname{Tr}(R_1R_2^{-1})= & \ (2e^{i\frac{\theta_2-\theta_1}{3}}+e^{-2i\frac{\theta_2-\theta_1}{3}}) \\
		&-4\sin^2(\phi)\sin(\theta_1/2)\sin(\theta_2/2)e^{-i\frac{\theta_2-\theta_1}{6}};
		\end{aligned}
		\end{equation}
		\item If their mirrors are disjoint and at distance $l$ from each other,
		\begin{equation}
		\begin{aligned}
		\operatorname{Tr}(R_1R_2^{-1})= & \ (2e^{i\frac{\theta_2-\theta_1}{3}}+e^{-2i\frac{\theta_2-\theta_1}{3}}) \\
		&+4\sinh^2(l/2)\sin(\theta_1/2)\sin(\theta_2/2)e^{-i\frac{\theta_2-\theta_1}{6}}.
		\end{aligned}
		\end{equation}
	\end{enumerate}
\end{Lemm}

\begin{proof}
		We use the normalization of Section \ref{reflnorm} and Lemma \ref{reflnormlemma}. We compute the matrix products and the formulas follow from trigonometric manipulations.
\end{proof}

\begin{Rema}
	These traces are only defined up to multiplication by $\omega$.
\end{Rema}

A geometrical consequence of the previous formulas may be given using the curve $\Delta$.

\begin{Prop}\label{tangent}
	Let $R_1,R_2$ be two complex reflections with rotation factors $\eta_1=e^{i\theta_1},\eta_2=e^{i\theta_2}$. Then $\operatorname{Tr}(R_1R_2^{-1})$ is on the tangent line to the deltoid $\Delta$, at the point of parameter $(\theta_2-\theta_1)/3$.
\end{Prop}

\begin{proof}
	Let us compute the tangent vector to $\Delta$ at the point of parameter $\theta$.
	$$\frac{d}{d\theta}(2e^{i\theta}+e^{-2i\theta})=2i(e^{i\theta}-e^{-2i\theta})=-4\sin(3\theta/2)e^{-i\theta/2}.$$
	Therefore, the tangent line to $\Delta$ at the point of parameter $\theta$ is the line $\{2e^{i\theta}+e^{-2i\theta}+te^{-i\theta/2},t\in\RR\}$. The proposition follows from Lemma \ref{traceformula}.
\end{proof}

\begin{Rema}
	We notice that the two numbers $-4\sin^2(\phi)\sin(\theta_1/2)\sin(\theta_2/2)$ and $4\sinh^2(l/2)\sin(\theta_1/2)\sin(\theta_2/2)$ have opposite signs and therefore parameterize points that lie on the two different halves of the tangent line. Consequently, the fact that the mirrors of $R_1$ and $R_2$ intersect or not in $\mathbb{H}_\CC^2$ imposes the half-tangent line on which this trace is.
\end{Rema}

\begin{Lemm}
	Any point inside $\Delta$ is at the intersection of exactly three lines tangent to $\Delta$. Besides, the three angles parameterizing the three tangency points have sum $0$.
\end{Lemm}

\begin{proof}
	In the proof of Proposition \ref{tangent}, we showed that the line tangent to $\Delta$ at the point of parameter $\theta$ is $\{2e^{i\theta}+e^{-2i\theta}+te^{-i\theta/2},t\in\RR\}$. Note that this set is equal to $\{e^{i\theta}+te^{-i\theta/2},t\in\RR\}$; indeed,
	$$2e^{i\theta}+e^{-2i\theta}+te^{-i\theta/2}=e^{i\theta}+e^{-i\theta/2}(2\cos(3\theta/2)+t).$$ 
	Let us now consider a point $e^{i\theta}+te^{-i\theta/2}$ on this tangent line. We are searching for which angles $\alpha$ this point is also on the tangent line at the point of parameter $\alpha$, i.e.\ for which values of $\alpha$ there exists a real number $s$ verifying $e^{i\theta}+te^{-i\theta/2}=e^{i\alpha}+se^{-i\alpha/2}$. We are therefore looking for the values of $\alpha$ for which $$e^{i(\theta+\alpha/2)}+te^{i(\alpha-\theta)/2}-e^{3i\alpha/2}\in\RR,$$ which leads us to solving the equation 
	$$\sin\left(\theta+\frac{\alpha}{2}\right) +t\sin\left(\frac{\alpha-\theta}{2}\right) -\sin\left(\frac{3\alpha}{2}\right) =0,$$
	or equivalently
	$$\sin\left(\frac{\alpha-\theta}{2}\right)\left(t-2\cos\left(\frac{\theta}{2}+\alpha\right)\right)=0.$$
	The solutions are the numbers $\alpha$ such that $\alpha=\theta$ or $\cos\left(\frac{\theta}{2}+\alpha\right)=\frac{t}{2}$. The three possible values for $\alpha$ are
	$$\alpha_1=\theta, \ \alpha_2=\gamma-\frac{\theta}{2}, \ \alpha_3=-\gamma-\frac{\theta}{2}$$
	with $\gamma=\arccos(t/2)\in[0,\pi]$. They satisfy $\alpha_1+\alpha_2+\alpha_3=0$.
\end{proof}

If we know the value of $\operatorname{Tr}(R_1R_2^{-1})$, then there are exactly three lines through it that are tangent to $\Delta$, so three possible values for $(\theta_2-\theta_1)/3$. See Figure \ref{tangents}.

\begin{figure}
	\centering
	\scalebox{0.3}{\includegraphics{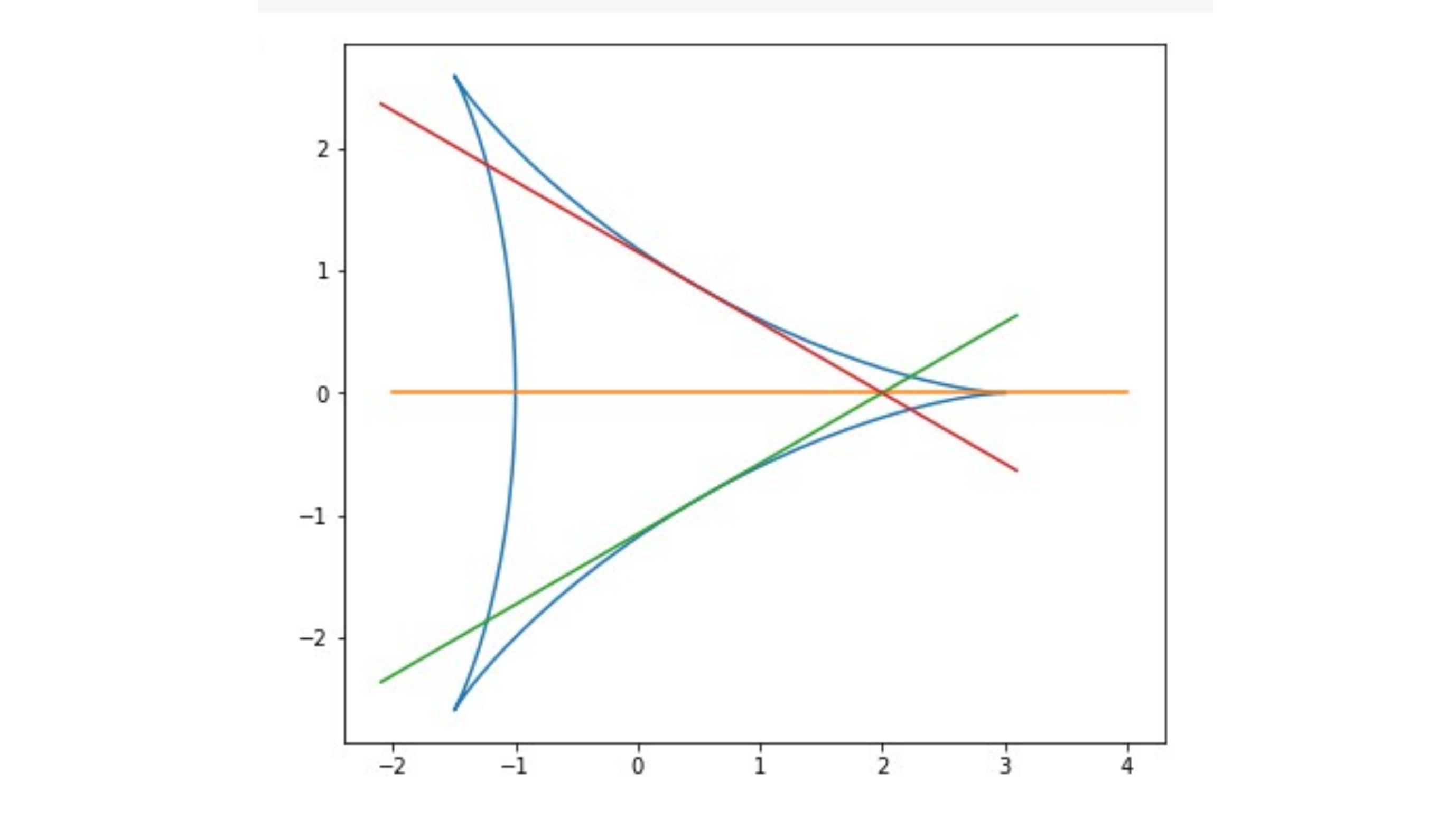}}
	\caption{The three tangent lines through a point. \label{tangents}}
\end{figure}

We will now interpret this result in terms of conjugacy classes in $\operatorname{SU}(2,1)$.

\begin{Coro}\label{tgfoot}
	Let $T$ be the trace of a regular elliptic element, and $\lambda_-$ its eigenvalue of negative type. Then, $T$ lies on a tangent line to $\Delta$ whose foot is the number $2\lambda_-+\lambda_-^{-2}$.
\end{Coro}

\begin{proof}
	Let $A\in\operatorname{SU}(2,1)$ be a regular elliptic element with trace $T$. We may assume, up to conjugation in $\operatorname{SU}(2,1)$, that $A$ fixes the origin. We know that there exist two complex reflections $R_1,R_2$ with rotation factors $\eta_1=e^{i\theta_1},\eta_2=e^{i\theta_2}$ such that $A=R_1R_2^{-1}$ (their mirrors are the two stable complex lines of $A$). Consequently, we have $\lambda_-=e^{i\frac{\theta_2-\theta_1}{3}}$. We conclude thanks to the formula in Lemma \ref{traceformula} and Proposition \ref{tangent}.
\end{proof}

\begin{Rema}\label{vapneg}
	\begin{itemize}
		\item We showed that if $A=R_1R_2^{-1}$ with $R_1,R_2$ of rotation factors $\eta_1=e^{i\theta_1},\eta_2=e^{i\theta_2}$, then the negative type eigenvalue of $A$ is equal to $e^{i\frac{\theta_2-\theta_1}{3}}$.
		\item The trace $T$ of an element in $\operatorname{SU}(2,1)$ is enough to compute its eigenvalues; however, determining its conjugacy class amounts to choosing which one is of negative type. Consider the three tangent lines meeting in $T$. Corollary \ref{tgfoot} tells us that choosing the conjugacy class is equivalent to choosing one tangent line among the three.
	\end{itemize}
\end{Rema}

We can now describe more precisely the triples $(R_1,R_2,R_3)$ that decompose a given spherical decomposable triple $(A,B,C)$. Indeed, we can obtain the possible angles $(\theta_1,\theta_2,\theta_3)$ of the complex reflections, and the angles $(\phi_1,\phi_2,\phi_3)$ at which the mirrors meet. Let us start with three regular elliptics $A,B,C$ whose product is the identity. We assume that the three conjugacy classes are known; we call $\lambda_A,\lambda_B,\lambda_C$ the three eigenvalues of negative type.
\begin{itemize}
	\item According to Theorem \ref{caractdecomp}, a spherical decomposition exists if and only if we have $(\lambda_A\lambda_B\lambda_C)^3=1$.
	\item Using Remark \ref{vapneg}, we obtain $\lambda_A=e^{i\frac{\theta_2-\theta_1}{3}}$, $\lambda_B=e^{i\frac{\theta_3-\theta_2}{3}}$, $\lambda_C=e^{i\frac{\theta_1-\theta_3}{3}}$ which gives us $(\theta_1,\theta_2,\theta_3)=(\theta,\theta+\beta,\theta+\gamma)$ with $\beta,\gamma$ fixed by $\lambda_A,\lambda_B,\lambda_C$ and $\theta$ a free parameter. 
	\item We now have
	$$\operatorname{Tr}(A)=(2e^{i\frac{\beta}{3}}+e^{-2i\frac{\beta}{3}})-4\sin^2(\phi_3)\sin(\theta/2)\sin((\theta+\beta)/2)e^{-i\frac{\beta}{6}}.$$
	Writing the analogous formulas for the traces of $B$ and $C$, then taking the real part in these equations, leads to the following expressions of the complex angles in terms of $\theta$:
	$$\sin^2(\phi_1)=\frac{2\cos((\beta-\gamma)/3)+\cos(2(\beta-\gamma)/3)-\operatorname{Re}(\operatorname{Tr}(B))}{2(\cos((\beta-\gamma)/2)-\cos(\theta+(\beta-\gamma)/2))\cos((\beta-\gamma)/6)},$$
	$$\sin^2(\phi_2)=\frac{2\cos(\gamma/3)+\cos(2\gamma/3)-\operatorname{Re}(\operatorname{Tr}(C))}{2(\cos(\gamma/2)-\cos(\theta+\gamma/2))\cos(\gamma/6)},$$
	$$\sin^2(\phi_3)=\frac{2\cos(\beta/3)+\cos(2\beta/3)-\operatorname{Re}(\operatorname{Tr}(A))}{2(\cos(\beta/2)-\cos(\theta+\beta/2))\cos(\beta/6)}.$$
	In particular, the right hand sides must belong to $]0,1]$, which gives us three domains of definition for $\theta$. The three domains have a common intersection, since the decomposability is assumed in the first point.
	\item At the end, we obtain a domain $D\subset[0,2\pi]$ for $\theta$, and the choice of such a $\theta$ defines $\theta_1=\theta,\theta_2=\theta+\beta,\theta_3=\theta+\gamma$ and $\phi_1,\phi_2,\phi_3$.
\end{itemize}

\section{Character varieties and parameterizations}\label{param}

In this section, we relate the decomposability property to the character variety of $\mathbb{F}_2$ in $\operatorname{SU}(2,1)$. A general description is given in \cite{Aco}. In order to get a complete parameterization, we shall describe the triangle defined by the mirrors of $R_1,R_2,R_3$.

\subsection{Complex triangles and their reflection group}\label{reflgrp}

We study triangles formed by triples of complex lines (see Brehm \cite{Bre} and Pratoussevitch \cite{Pra}). We use the conventions and notations of the latter.

\begin{Defi}\label{angles}
	Let $C_1,C_2,C_3$ be three complex lines intersecting pairwise, with polar vectors $c_1,c_2,c_3$. Let $\phi_k$ be the complex angle between $C_{k-1}$ and $C_{k+1}$, and $r_k=\cos(\phi_k)$, $s_k=\sin(\phi_k)$. 
	\begin{enumerate} [i)]
		\item The triple $(C_1,C_2,C_3)$ is called a $(\phi_1,\phi_2,\phi_3)$-triangle;
		\item The \textit{angular invariant} of the complex triangle is the number
		$$\alpha=\arg\left(\prod_{k=1}^{3}\langle c_{k-1},c_{k+1}\rangle\right),$$
		which does not depend on the choices of $c_1,c_2,c_3$.
	\end{enumerate}
\end{Defi}

The angular invariant describes the set of $(\phi_1,\phi_2,\phi_3)$-triangles up to $\operatorname{PU}(2,1)$.

\begin{Prop}\label{triangleclassif}
	\emph{\cite{Pra}} A $(\phi_1,\phi_2,\phi_3)$-triangle is determined uniquely up to holomorphic isometry by its angular invariant $\alpha$.
	For $\alpha\in[0,2\pi[$, there exists a $(\phi_1,\phi_2,\phi_3)$-triangle with angular invariant $\alpha$ if and only if $\alpha$ belongs to the interval defined by:
	$$\cos(\alpha)<\frac{r_1^2+r_2^2+r_3^2-1}{2r_1r_2r_3}.$$
\end{Prop}

An explicit formula for the matrices of the three complex reflections associated to a $(\phi_1,\phi_2,\phi_3)$-triangle of invariant $\alpha$ is given by Lemma \ref{reflnormlemma}. From now on, we fix angles $\phi_1,\phi_2,\phi_3$, with $r_i=\cos(\phi_i),s_i=\sin(\phi_i)$, and we suppose that the condition of Proposition \ref{triangleclassif} is satisfied. We use the normalization of Section \ref{reflnorm}. Note that the triple $(R_1,R_2,R_3)$ is irreducible by construction.

\begin{Rema}
	\cite{Pra} When $\alpha=\pi$, the three vertices of the triangle lie in a copy of $\mathbb{H}_{\RR}^2$. The corresponding reflection group is then embedded in $\operatorname{PO}(2,1)$, as described in the introduction. When $\cos(\alpha)=\frac{r_1^2+r_2^2+r_3^2-1}{2r_1r_2r_3}$, the three vertices of the triangle coincide.
\end{Rema}

The reflexion matrices will be useful to give explicit parameterizations of subsets of character varieties for $\mathbb{F}_2$ or triangle groups.

\subsection{Character variety of $\mathbb{F}_2$}\label{charvarsection}

Following \cite{Aco}, we review a few notions about $\operatorname{SL}_n(\CC)$-character varieties, in order to transform our results into descriptions of subsets of the character variety of $\mathbb{F}_2$ into $\operatorname{PU}(2,1)$. Let $\Gamma$ be a finitely generated group. Recall that the $\operatorname{SL}_n(\CC)$-character variety of $\Gamma$ is the algebraic quotient
$$\chi_{\operatorname{SL}_n(\CC)}(\Gamma)=\operatorname{Hom}(\Gamma,\operatorname{SL}_n(\CC))//\operatorname{SL}_n(\CC),$$ 
where the quotient is with respect to the action  by conjugation of $\operatorname{SL}_n(\CC)$ on $\operatorname{Hom}(\Gamma,\operatorname{SL}_n(\CC))$.

It is a classical result (see for instance Theorems 1.3 and 3.3 of \cite{Pro}) that the ring of invariant functions of $\operatorname{Hom}(\Gamma,\operatorname{SL}_n(\CC))$ is generated by the trace functions $\tau_\gamma:\rho\mapsto\operatorname{Tr}(\rho(\gamma))$, for $\gamma$ in a finite subset $\{\gamma_1,\ldots,\gamma_k\}$ of $\Gamma$. Consequently, $\chi_{\operatorname{SL}_n(\CC)}(\Gamma)$ is isomorphic to the image of $(\tau_{\gamma_1},\ldots,\tau_{\gamma_k}):\operatorname{Hom}(\Gamma,\operatorname{SL}_n(\CC))\to\CC^k$.

We recall that the character of $\rho\in\operatorname{Hom}(\Gamma,\operatorname{SL}_n(\CC))$ is the function $\chi_\rho:\Gamma\to\CC$ given by $\chi_\rho(\gamma)=\operatorname{Tr}(\rho(\gamma))$. Now, let $\rho,\rho'\in\operatorname{Hom}(\Gamma,\operatorname{SL}_n(\CC))$ be two semi-simple representations. Then $\chi_\rho=\chi_\rho'$ if and only if $\rho$ and $\rho'$ are conjugate (see for example Theorem 1.28 of \cite{LubMag}). This result enables us to identify a semi-simple representation in the character variety with its character.

We now focus on the $\operatorname{SL}_3(\CC)$-character variety of $\mathbb{F}_2=\langle a,b\rangle$. It has already been parameterized, as stated in the following result, which is proven in \cite{Law}:

\begin{Theo}\label{charvar}
	The character variety $\chi_{\operatorname{SL}_3(\CC)}(\mathbb{F}_2)$ is isomorphic to the algebraic set $V$ of $\CC^9$ which is the image of $\operatorname{Hom}(\mathbb{F}_2,\operatorname{SL}_3(\CC))$ by the trace functions of the elements $a,b,ab,ab^ {-1},a^{-1},b^{-1},b^{-1}a^{-1},ba^ {-1},[a,b]$. Furthermore, there exist two polynomials $P,Q\in\CC[X_1,\ldots,X_8]$ such that $(x_1,\ldots,x_8)\in V$ if and only if $x_9^2-Q(x_1,\ldots,x_8)x_9+P(x_1,\ldots,x_8)=0$.
\end{Theo}

Remaining in the spirit of \cite{Aco}, we define the $\operatorname{SU}(2,1)$-character variety of $\mathbb{F}_2$ as the subset of $\chi_{\operatorname{SL}_3(\CC)}(\mathbb{F}_2)$ given by:
$$\chi_{\operatorname{SU}(2,1)}(\mathbb{F}_2):=\{\chi\in\chi_{\operatorname{SL}_3(\CC)}(\mathbb{F}_2) \ | \ \exists\rho\in\operatorname{Hom}(\mathbb{F}_2,\operatorname{SU}(2,1)),\chi=\chi_\rho\}.$$

Note that if $\rho,\rho'\in\operatorname{Hom}(\mathbb{F}_2,\operatorname{SU}(2,1))$ are conjugate in $\operatorname{SL}_3(\CC)$ and Zariski-dense, they are in fact conjugate in $\operatorname{SU}(2,1)$. We may therefore adapt Theorem \ref{charvar} to $\operatorname{SU}(2,1)$, using that for $A\in\operatorname{SU}(2,1)$, $\operatorname{Tr}(A^{-1})=\overline{\operatorname{Tr}(A)}$.

\begin{Coro}\label{charvarcoro}
	A Zariski-dense representation of $\mathbb{F}_2$ into $\operatorname{SU}(2,1)$ is uniquely determined up to conjugacy by the following traces: $\operatorname{Tr}(A),\operatorname{Tr}(B),$ $\operatorname{Tr}(AB),$ $\operatorname{Tr}(A^{-1}B),\operatorname{Tr}([A,B])$. Besides, $\operatorname{Tr}([A,B])$ is a root of a degree 2 polynomial whose coefficients are polynomials in the first four traces (the one that appears in Theorem \ref{charvar}).
\end{Coro}

\begin{Defi}
	\begin{itemize}
		\item We denote by $\chi^{sd}$, respectively $\chi^{pd}$, the subset of the character variety $\chi_{\operatorname{SU}(2,1)}(\mathbb{F}_2)$ constituted by the characters of spherical, respectively pure, decomposable representations.
		\item Let $\cC_1,\cC_2,\cC_3$ be fixed conjugacy classes in $\operatorname{SU}(2,1)(\mathbb{F}_2)$. We call \textit{relative component of} $\chi_{\operatorname{SU}(2,1)}$ \textit{with respect to} $\cC_1,\cC_2,\cC_3$ the subset constituted by the characters of representations $\rho$ that verify $\rho(a)\in\cC_1,\rho(b)\in\cC_2,\rho(ab)\in\cC_3$. We denote it by $\chi_{\cC_1,\cC_2,\cC_3}$.
	\end{itemize}
\end{Defi}

The discussions of the previous sections now enable us to parameterize the relative components of $\chi^{sd}$. We recall that $\theta$ is the rotation angle of the complex reflection $R_1$, which belongs to a domain $D_{\cC_1,\cC_2,\cC_3}$ entirely determined by the data of $\cC_1,\cC_2,\cC_3$ (see the end of Section \ref{tracesanddecomp}). The number $\alpha$ is the angular invariant of the complex triangle, which belongs for each $\theta\in D_{\cC_1,\cC_2,\cC_3}$ to the interval $I(\theta)=\{\alpha\in[0,2\pi] \ | \ \cos(\alpha)<\delta(\theta)\}$, where $\delta(\theta)$ is the bound appearing in Proposition \ref{triangleclassif}. The notion of (ir)reducibility is in the sense of Definition \ref{irred}.

\begin{Theo}
	Let $\cC_1,\cC_2,\cC_3$ be fixed elliptic conjugacy classes in $\operatorname{SU}(2,1)$. We have a homeomorphism between the relative component $\chi^{sd}_{\cC_1,\cC_2,\cC_3}$ and the following subset of $[0,2\pi]^2$:
	$$S_{\cC_1,\cC_2,\cC_3}:=\{(\theta,\alpha)\in[0,2\pi]^2 \ | \ \theta\in D_{\cC_1,\cC_2,\cC_3}, \ \alpha\in I(\theta)\}.$$ Besides, $\chi^{sd}_{\cC_1,\cC_2,\cC_3}$ contains reducible representations in its closure.
\end{Theo}

\begin{proof}
	Start with a triple $(A,B,C)$ corresponding to a representation in $\chi^{sd}_{\cC_1,\cC_2,\cC_3}$. We have a continuous map sending $(A,B,C)$ to the triple of fixed points $(P_A,P_B,P_C)\in(\mathbb{H}_{\CC}^2)^3$. This triple and Theorem \ref{caractdecomp} give us a unique triangle of complex lines, together with a decomposition, and thereby a pair $(\theta,\alpha)\in S_{\cC_1,\cC_2,\cC_3}$ (see the discussion at the end of Section \ref{tracesanddecomp}).
	
	Conversely, if $(\theta,\alpha)\in S_{\cC_1,\cC_2,\cC_3}$, then the formulas at the end of Section \ref{tracesanddecomp} yield continuously a unique triple of complex angles, and therefore a triangle of complex lines realizing the decomposition of a triple in $\chi^{sd}_{\cC_1,\cC_2,\cC_3}$.
\end{proof}

The last part of the theorem is given by the more precise following result.

\begin{Lemm}\label{deform}
	Let $\rho:\mathbb{F}_2\to\operatorname{SU}(2,1)$ be an irreducible spherical decomposable representation. Then there exists a spherical decomposable deformation $(\rho_t)_{0\leq t\leq1}$ that preserves the conjugacy classes, such that $\rho_0=\rho$ and $\rho_1$ is a reducible decomposable representation (one whose mirrors have a triple intersection point).
\end{Lemm}

\begin{proof}
	Let us denote $$\delta=\frac{r_1^2+r_2^2+r_3^2-1}{2r_1r_2r_3}.$$
	Our representation corresponds to the data of $\theta_1$, $\theta_2$, $\theta_3$, $\alpha$, $\phi_1$, $\phi_2$, $\phi_3$, with $\cos(\alpha)<\delta.$
	Fixing the conjugacy classes of $A,B,C$ amounts to fixing $\operatorname{Tr}(A)$, $\operatorname{Tr}(B)$, $\operatorname{Tr}(C)$ as well as the values of $(\theta_i-\theta_j)/3$. Besides, the mirrors have a triple intersection point if and only if 
	$\cos(\alpha)=\delta.$
	We may therefore keep the values of $\phi_1$, $\phi_2$, $\phi_3$, $\theta_1$, $\theta_2$, $\theta_3$ fixed, and consider a path $(\alpha_t)_t$ with $\alpha_0=\alpha$ and $\cos(\alpha_1)=\delta$. The corresponding path of representations takes our initial representation to a decomposable representation whose mirrors have a triple intersection point, without changing the conjugacy classes, which concludes the proof.
\end{proof}

This result also holds for any hyperbolic decomposable configuration, replacing the angles $\phi_i$ by distances between complex lines $l_i$, and $\cos(\phi_i),\sin(\phi_i)$ by $\cosh(l_i),\sinh(l_i)$:

\begin{Lemm}\label{deform'}
	Let $\rho:\mathbb{F}_2\to\operatorname{SU}(2,1)$ be an irreducible hyperbolic decomposable representation. Then there exists a hyperbolic decomposable deformation $(\rho_t)_{0\leq t\leq1}$ that preserves the conjugacy classes, such that $\rho_0=\rho$ and $\rho_1$ is a reducible decomposable representation (one whose three mirrors have a common orthogonal complex line).
\end{Lemm}

\section{The momentum map}\label{momentummap}

\subsection{Momentum map and reducible skeleton}

The \textit{multiplicative Horn problem} is the following question: given two fixed conjugacy classes $\cC_1,\cC_2$ in $\operatorname{PU}(2,1)$, which conjugacy classes are obtained by computing the products of two matrices taken in $\cC_1,\cC_2$? In \cite{Pau}, Paupert proposes a method to describe precisely the image of the following map:

\begin{Defi}
	Let $\cG$ be the set of conjugacy classes in $\operatorname{PU}(2,1)$ and $\cC_1,\cC_2$ two conjugacy classes in $\operatorname{PU}(2,1)$. We denote by $[g]$ the conjugacy class of an element $g$. The \textit{momentum map for $\cC_1,\cC_2$} is defined as
	$$\begin{array}{ccccc}
	\mu_{\cC_1,\cC_2} & : & \cC_1\times \cC_2 & \to & \cG \\
	& & (A,B) & \mapsto & [AB]. \\
	\end{array}$$
\end{Defi}

We first review some of Paupert's results in \cite{Pau}. Using the notations of Definition \ref{anglepair}, he describes the image of the restriction $\mu_{\cC_1,\cC_2}\big|_{\mu^{-1}_{\cC_1,\cC_2}(\mbox{elliptics})}$. When there is no ambiguity regarding the choice of the two conjugacy classes, we simply call this restriction $\mu$. 

\begin{Rema}
	Notice that the fiber of $\mu$ above $\cC_3$ modulo conjugation is identified with the relative character variety $\chi_{\cC_1,\cC_2,\cC_3}$.
\end{Rema}

If $A$ is conjugate to $E(\alpha_1,\alpha_2)$ (see Remark \ref{ellip}) and $B$ is conjugate to $E(\beta_1,\beta_2)$, the angle pair of the product $AB$ (which is elliptic) may be drawn in a half-square of size $2\pi$ whose sides of length $2\pi$ are identified, forming a Möbius band. From now on, we fix two regular elliptic conjugacy classes $\cC_1,\cC_2$. We describe the image of $\mu$ when the source is restricted to \textit{reducible pairs}.

\begin{Defi}
	\cite{Pau} Let $(A,B)\in\cC_1\times \cC_2$ be a pair of regular elliptic elements. We say that this pair is
	\begin{itemize}
		\item \textit{totally reducible} when $A$ and $B$ are diagonalizable in the same basis (i.e.\ they commute). This means that they have the same fixed point and the same stable complex lines.
		\item \textit{spherical reducible} when $A$ and $B$ have one common eigenvector of negative type. This means that they have the same fixed point in $\mathbb{H}_\CC^2$.
		\item \textit{hyperbolic reducible} when $A$ and $B$ have one common eigenvector of positive type. This means that they have one stable complex line in common.
	\end{itemize}
\end{Defi}

The following theorem summarizes Paupert's results concerning the image of reducible pairs.

\begin{Theo}
	\emph{\cite{Pau}} Let $\cC_1,\cC_2$ be a pair of regular elliptic conjugacy classes, represented by the angle pairs $(\alpha_1,\alpha_2)$ and $(\beta_1,\beta_2)$.
	\begin{enumerate} [i)]
		\item The image of $\mu$ restricted to the totally reducible pairs is represented by the following two unordered angle pairs: 
		$$\{\alpha_1+\beta_1,\alpha_2+\beta_2\}, \ \{\alpha_1+\beta_2,\alpha_2+\beta_1\}.$$
		These two points are called the \textit{totally reducible vertices}.
		\item The image of $\mu$ restricted to the spherical reducible pairs is represented by the segment of slope $-1$ that connects the totally reducible vertices. We call it the \textit{spherical reducible segment}.
		\item The image of $\mu$ restricted to the hyperbolic reducible pairs is represented by segments of slope $2$ or $1/2$ issued from the totally reducible vertices. We call them the \textit{hyperbolic reducible segments}.
	\end{enumerate}
\end{Theo}

\begin{Rema}
	\begin{itemize}
		\item The totally reducible vertices are given by the matrix products $E(\alpha_1,\alpha_2)E(\beta_1,\beta_2)$ and $E(\alpha_1,\alpha_2)E(\beta_2,\beta_1)$.
		\item The reducible segments may be disconnected in the half-square.
	\end{itemize}
\end{Rema}

\begin{Defi}
	The union of the totally reducible vertices and the reducible segments is called the \textit{reducible skeleton} of $\mu$ for $\cC_1,\cC_2$. A fiber above the spherical, respectively hyperbolic, reducible skeleton is called a \textit{spherical}, respectively \textit{hyperbolic}, \textit{reducible fiber}. See Figures \ref{moment} and \ref{moment2} for examples.
\end{Defi}

\begin{Rema}
	In his work, Paupert describes $\operatorname{Im}(\mu)$ as the union of the reducible skeleton and some connected components of its complement. We will not use those results here, and we refer the reader to \cite{Pau} for more details concerning this aspect.
\end{Rema}

\begin{minipage}{0.4\textwidth}
	\begin{figure}[H]
		\includegraphics[scale=0.28]{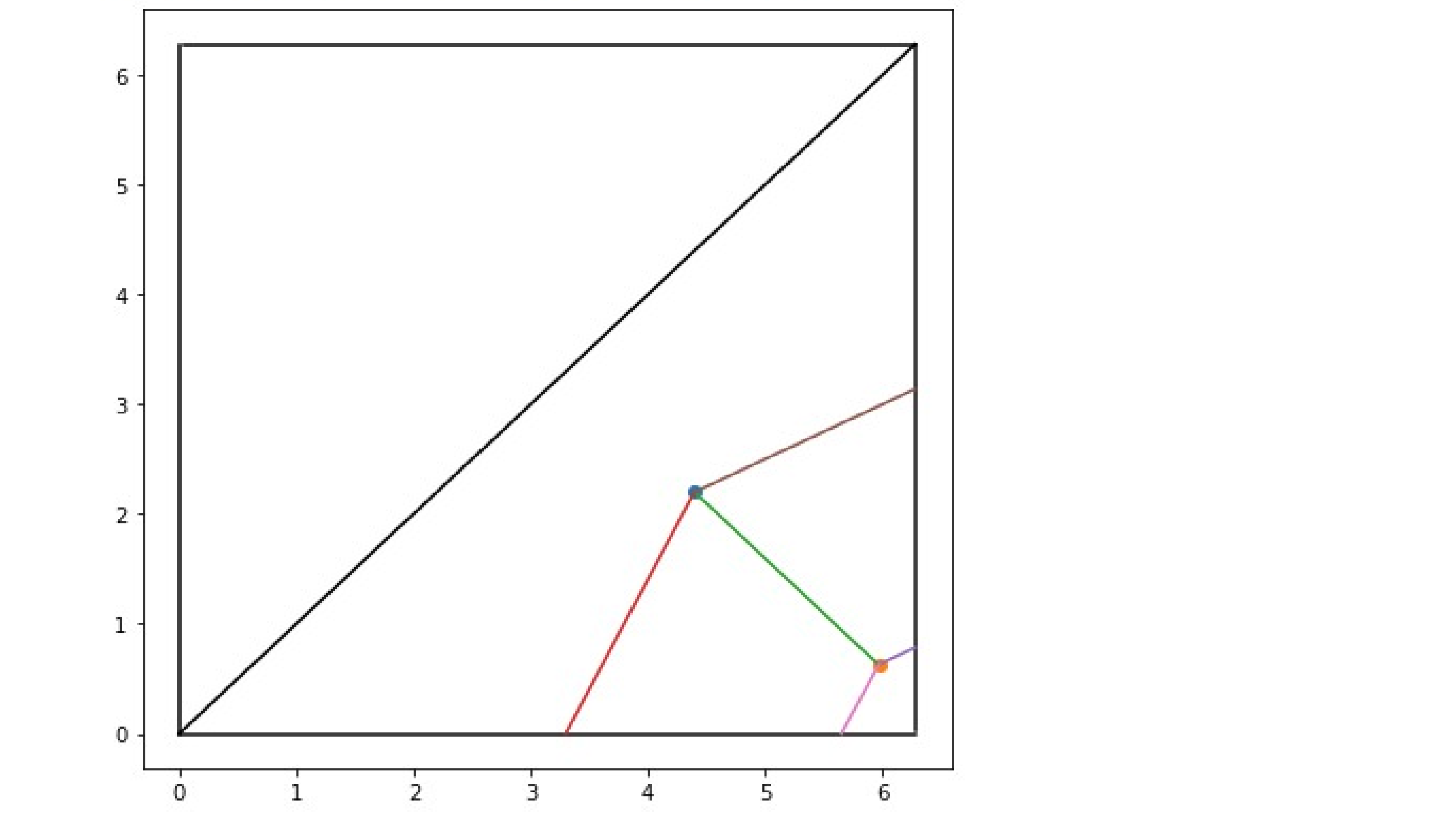}
		\caption{The reducible skeleton for $A\sim(3\pi/2,\pi)$ and $B\sim(6\pi/5,2\pi/5)$.}
		\label{moment}
	\end{figure}
\end{minipage}
\hspace{4ex} 
\begin{minipage}{0.4\textwidth}
	\begin{figure}[H]
		\includegraphics[scale=0.28]{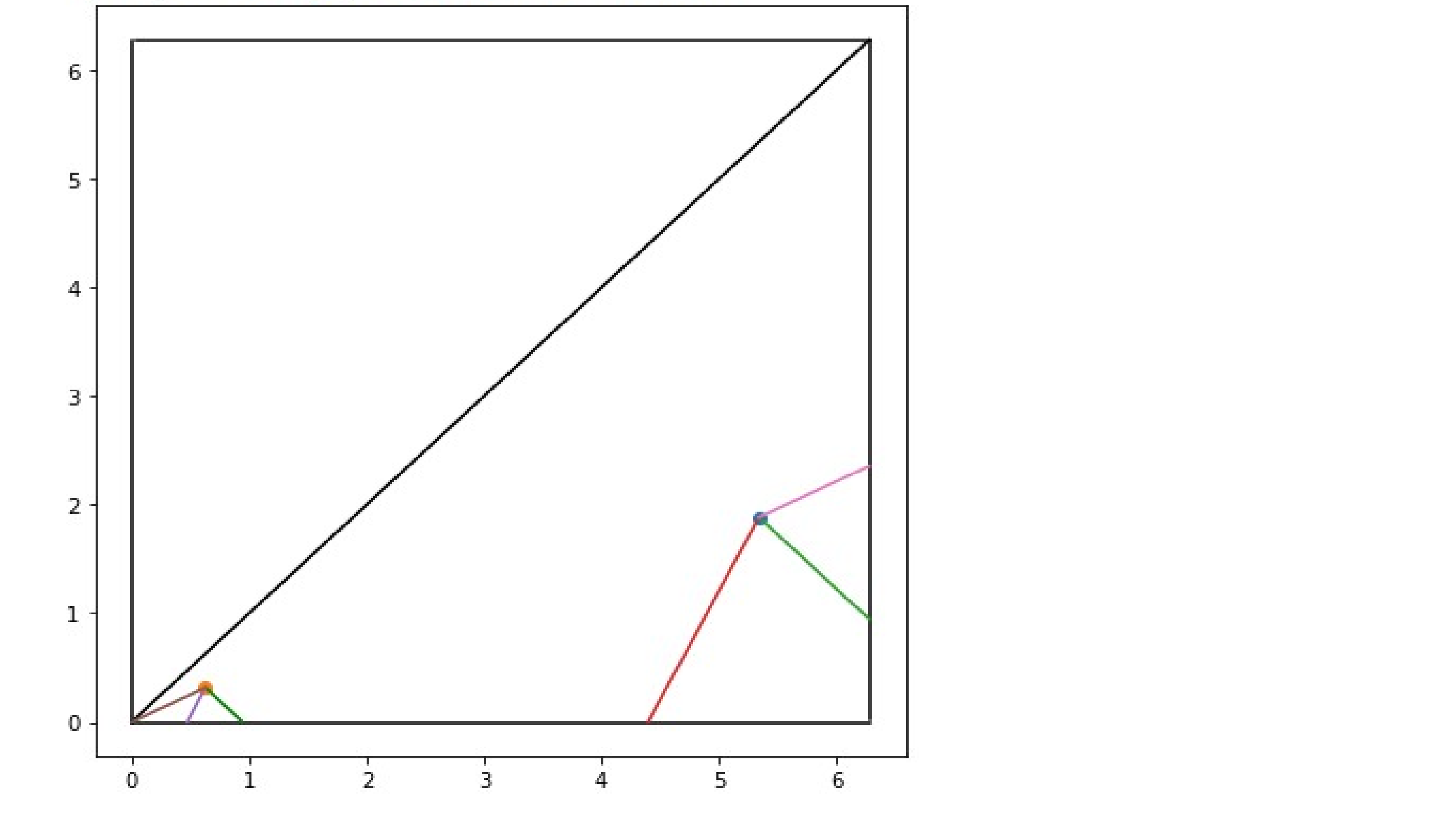}
		\caption{The reducible skeleton for $A\sim(8\pi/5,6\pi/5)$ and $B\sim(\pi,\pi/2)$.}
		\label{moment2}
	\end{figure}
\end{minipage}

\subsection{Decomposable representations and reducible fibers}

We now relate decomposable pairs and reducible fibers.

\begin{Theo}\label{fiber}
	Consider a spherical, respectively hyperbolic, reducible fiber of $\mu$ that contains an irreducible pair. Then this fiber is connected and contains only spherical, respectively hyperbolic, decomposable pairs. Besides, it contains a unique reducible pair up to conjugation, which is also decomposable.
\end{Theo}

To prove this theorem, we need a couple of lemmas concerning irreducible and reducible pairs.

\begin{Lemm}\label{squelette--decomp}
	Let $\cC_1,\cC_2$ be two regular elliptic conjugacy classes and $(A,B)\in\cC_1\times\cC_2$ be an irreducible pair. Then, the angle pair of the product $AB$ lies on the spherical, respectively hyperbolic, reducible skeleton if and only if the triple $(A,B,(AB)^{-1})$ is spherical, respectively hyperbolic, decomposable.
\end{Lemm}

Lemmas \ref{deform} and \ref{deform'} prove the "if" part of the statement, but we give here a direct proof of the equivalence.

\begin{proof}
	Let us write the diagonal representatives of the conjugacy classes of $A,B,AB$ in $\operatorname{SU}(2,1)$:
	$$A\sim\begin{pmatrix}
	e^{i\alpha_1} & 0 & 0 \\
	0 & e^{i\alpha_2} & 0 \\
	0 & 0 & e^{i\alpha_3}
	\end{pmatrix}, B\sim\begin{pmatrix}
	e^{i\beta_1} & 0 & 0 \\
	0 & e^{i\beta_2} & 0 \\
	0 & 0 & e^{i\beta_3}
	\end{pmatrix},AB\sim\begin{pmatrix}
	e^{i\gamma_1} & 0 & 0 \\
	0 & e^{i\gamma_2} & 0 \\
	0 & 0 & e^{i\gamma_3}
	\end{pmatrix},$$
	with $\alpha_1+\alpha_2+\alpha_3\equiv\beta_1+\beta_2+\beta_3\equiv\gamma_1+\gamma_2+\gamma_3\equiv0 \ [2\pi]$. The conjugacy classes of the three elements are represented in $\operatorname{PU}(2,1)$ by the following angle pairs:
	$$A\sim\{\alpha_1-\alpha_3,\alpha_2-\alpha_3\},B\sim\{\beta_1-\beta_3,\beta_2-\beta_3\},AB\sim\{\gamma_1-\gamma_3,\gamma_2-\gamma_3\}.$$
	The two totally reducible configurations associated to the pair $(A,B)$ are represented by the two angle pairs
	$$\{\alpha_1-\alpha_3+\beta_1-\beta_3,\alpha_2-\alpha_3+\beta_2-\beta_3\},\{\alpha_1-\alpha_3+\beta_2-\beta_3,\alpha_2-\alpha_3+\beta_1-\beta_3\}.$$
	We first examine the case where the angle pair of $AB$ lies on the spherical reducible skeleton. This means that it is on the line of slope $-1$ that connects the two totally reducible vertices. In other words, there exist two integers $k,k'$ such that
	$$\alpha_1-\alpha_3+\beta_1-\beta_3-(\gamma_1-\gamma_3)+2k\pi=-(\alpha_2-\alpha_3+\beta_2-\beta_3-(\gamma_2-\gamma_3)+2k'\pi)$$ or $$\alpha_1-\alpha_3+\beta_1-\beta_3-(\gamma_2-\gamma_3)+2k\pi=-(\alpha_2-\alpha_3+\beta_2-\beta_3-(\gamma_1-\gamma_3)+2k'\pi).$$
	Using $\alpha_1+\alpha_2+\alpha_3\equiv\beta_1+\beta_2+\beta_3\equiv\gamma_1+\gamma_2+\gamma_3\equiv0 \ [2\pi]$, both equations are equivalent to $3(\alpha_3+\beta_3-\gamma_3)\equiv0 \ [2\pi]$. This is equivalent to $$(e^{i\alpha_3}e^{i\beta_3}e^{-i\gamma_3})^3=1.$$ Since the triple $(A,B,(AB)^{-1})$ is irreducible, we may use Theorem \ref{caractdecomp}, which tells us that this equality is equivalent to the spherical decomposability of $(A,B,(AB)^{-1})$.
	
	The hyperbolic reducible case is similar.
\end{proof}

Now, each reducible fiber contains a unique conjugacy class of reducible pairs.

\begin{Lemm}\label{redunique}
	Let $\cC_1,\cC_2,\cC_3$ be three regular elliptic conjugacy classes. Let $(A_1,B_1)$ and $(A_2,B_2)$ be two hyperbolic or spherical reducible pairs in the fiber above $\cC_3$ of $\mu_{\cC_1,\cC_2}$. Then there exists $P\in\operatorname{PU}(2,1)$ such that $$(A_2,B_2)=(PA_1P^{-1},PB_1P^{-1}).$$
\end{Lemm}

\begin{proof}
	Let us start with the spherical reducible case. We may assume that $A_1,B_1,A_2,B_2$ are in block diagonal form:
	$$A_1=\begin{pmatrix}
	aA_1' & 0 \\
	0 & 1
	\end{pmatrix}, \ B_1=\begin{pmatrix}
	bB_1' & 0 \\
	0 & 1
	\end{pmatrix},A_2=\begin{pmatrix}
	aA_2' & 0 \\
	0 & 1
	\end{pmatrix}, \ B_2=\begin{pmatrix}
	bB_2' & 0 \\
	0 & 1
	\end{pmatrix}$$
	where $A_1',B_1',A_2',B_2'\in\operatorname{SU}(2)$. Now, $A_1,A_2\in\cC_1,B_1,B_2\in\cC_2,A_1B_1,A_2B_2\in\cC_3$ so 
	$$\operatorname{Tr}(A_1)=\operatorname{Tr}(A_2),\operatorname{Tr}(B_1)=\operatorname{Tr}(B_2),\operatorname{Tr}(A_1B_1)=\operatorname{Tr}(A_2B_2)$$
	and consequently
	$$\operatorname{Tr}(A_1')=\operatorname{Tr}(A_2'),\operatorname{Tr}(B_1')=\operatorname{Tr}(B_2'),\operatorname{Tr}(A_1'B_1')=\operatorname{Tr}(A_2'B_2').$$
	It is a classical fact that this implies that the pairs $(A_1',B_1')$ and $(A_2',B_2')$ are conjugate in $\operatorname{SL}_2(\CC)$. Now, this in turn implies that they are conjugate in $\operatorname{SU}(2)$ (see for example Lemma 3.6 in \cite{Aco}). As a consequence, the pairs $(A_1,B_1)$ and $(A_2,B_2)$ are conjugate in $\operatorname{PU}(2,1)$.
	
	The proof of the hyperbolic reducible case is similar.
\end{proof}

We are now able to prove Theorem \ref{fiber}.

\begin{proof}
	Let $F$ be a reducible fiber, and let $(A,B)\in F$ be an irreducible pair. According to Lemma \ref{squelette--decomp}, the triple $(A,B,(AB)^{-1})$ is decomposable.
	
	Now, Lemmas \ref{deform} and \ref{deform'} state that the decomposition of $(A,B,(AB)^{-1})$ may be deformed into the decomposition of a reducible triple with the same conjugacy classes; denote it $(A_0,B_0,(A_0B_0)^{-1})$. So we obtain the existence of the decomposition for the reducible pair $(A_0,B_0)\in F$. Besides, Lemma \ref{redunique} states that all the reducible pairs of the fiber are conjugate; therefore, the decomposition we just obtained for $(A_0,B_0,(A_0B_0)^{-1})$ may be globally conjugate in order to decompose any reducible pair of the fiber. This shows that the whole fiber is decomposable. In the process, we constructed a deformation path going from any triple to $(A_0,B_0,(A_0B_0)^{-1})$; this implies that the fiber is path-connected.
\end{proof}

\begin{Rema}
	We are now able to fully describe the reducible fibers of the momentum map. Two situations may occur:
	\begin{itemize}
		\item the fiber above a point of the reducible skeleton contains only reducible pairs: in this case, they are all conjugate, so the fiber is reduced to a point.
		\item the fiber above a point of the reducible skeleton contains an irreducible pair: in this case, this fiber is two-dimensional and fully decomposable.
	\end{itemize}
	In Proposition 2.5 of \cite{Pau}, Paupert proves that $\mu$ is locally surjective at an irreducible pair. Consequently, a reducible fiber may contain an irreducible pair only if its image is interior to the image of $\mu$, i.e.\ if the reducible skeleton does not bound $\operatorname{Im}(\mu)$.
\end{Rema}

\section{Triangle groups}\label{trigroups}

We now describe how to obtain parameterizations of components of character varieties for $(p,q,r)$-groups.

\subsection{Decomposing $(p,q,r)$-groups}

We first recall the following definitions.

\begin{Defi}
	Let $p,q,r\in\NN$ such that $1/p+1/q+1/r<1$. 
	\begin{itemize}
		\item We define the following group:
		$$\Gamma_{p,q,r}^{(2)}=\langle a,b,c \ | \ a^p=b^q=c^r=abc=1\rangle.$$
		\item We say that an elliptic conjugacy class $\cC$ of $\operatorname{PU}(2,1)$ has \textit{order $p$} when all its elements have order $p$.
	\end{itemize}
\end{Defi}

An elliptic conjugacy class of order $p$ is given by a choice of two $p$-roots of $1$. Since these roots are in finite number, there is a finite number of elliptic conjugacy classes of order $p$.

\begin{Rema}
	Each relative component of the character variety $\chi_{\operatorname{SU}(2,1)}(\Gamma_{p,q,r}^{(2)})$ is a fiber modulo global conjugation of some $\mu_{\cC_1,\cC_2}$ for $\cC_1,\cC_2$ elliptic conjugacy classes of order $p$ and $q$. We deduce the following result from the discussion of the previous section.
\end{Rema}

\begin{Coro}
	Let $p,q,r$ be fixed. Then $\chi^{pd}(\Gamma_{p,q,r}^{(2)})$ corresponds to the union of the fibers of $\mu_{\cC_1,\cC_2}$ for every $\cC_1$ and $\cC_2$ of order $p$ and $q$ above the points of the reducible skeleton whose conjugacy class is of order $r$.
\end{Coro}

\subsection{The $\RR$-Fuchsian component}\label{334}

\begin{Defi}
	We call \textit{$\RR$-Fuchsian component} of $\chi(\Gamma_{p,q,r}^{(2)})$ the relative component with respect to the following three angle pairs: $\{\frac{\pi}{p},-\frac{\pi}{p}\},$ $\{\frac{\pi}{q},-\frac{\pi}{q}\},$ $\{\frac{\pi}{r},-\frac{\pi}{r}\}$.
\end{Defi}

The $\RR$-Fuchsian component of $\chi(\Gamma_{p,q,r}^{(2)})$ is the one that contains the $\RR$-Fuchsian representation (i.e.\ the embedding $\Gamma_{p,q,r}^{(2)}\to\operatorname{PO}(2,1)\subset\operatorname{PU}(2,1)$), as well as representations decomposable in products of reflections of order $2$ (representations of $\Gamma_{p,q,r}$ as described in the introduction). In this component, the negative type eigenvalues of $A,B,C$ are all equal to 1; this imposes that the three angles of the complex reflections are equal, i.e.\ $\theta_1=\theta_2=\theta_3=\theta$. The results of this paper also have the following consequence:

\begin{Theo}\label{RFuchs}
	The $\RR$-Fuchsian component of $\chi(\Gamma_{p,q,r}^{(2)})$ is entirely composed of spherical decomposable representations. Besides, it contains a unique reducible point and is topologically a sphere. In particular, it is compact.
\end{Theo}

\begin{proof}
	In the real Fuchsian component, $A,B,C$ are in the following conjugacy classes:
	$$A\sim\begin{pmatrix}
	e^{2i\pi/p} & 0 & 0 \\
	0 & e^{-2i\pi/p} & 0 \\
	0 & 0 & 1 \\
	\end{pmatrix}, \ B\sim\begin{pmatrix}
	e^{2i\pi/q} & 0 & 0 \\
	0 & e^{-2i\pi/q} & 0 \\
	0 & 0 & 1 \\
	\end{pmatrix},$$
	$$C\sim\begin{pmatrix}
	e^{2i\pi/r} & 0 & 0 \\
	0 & e^{-2i\pi/r} & 0 \\
	0 & 0 & 1 \\
	\end{pmatrix}.$$
	The fact that the three eigenvalues of negative type are equal to $1$ implies that a spherical decomposition exists, with $e^{i\frac{\theta_2-\theta_1}{3}}=e^{i\frac{\theta_3-\theta_2}{3}}=e^{i\frac{\theta_1-\theta_3}{3}}=1$. This imposes $\theta_1=\theta_2=\theta_3=\theta$.
	
	Since $\operatorname{Tr}(A)=1+2\cos(2\pi/p),\operatorname{Tr}(B)=1+2\cos(2\pi/q),\operatorname{Tr}(C)=1+2\cos(2\pi/r)$, the formulas of Section \ref{tracesanddecomp} yield
	$$\sin^2(\phi_1)=\frac{1-\cos(2\pi/q)}{2\sin^2(\theta/2)}, \ \sin^2(\phi_2)=\frac{1-\cos(2\pi/r)}{2\sin^2(\theta/2)},$$
	$$\sin^2(\phi_3)=\frac{1-\cos(2\pi/p)}{2\sin^2(\theta/2)}.$$
	Therefore if $m=\operatorname{min}(p,q,r)$ then $\theta\in]2\pi/m,2\pi-2\pi/m[$, and the angular invariant $\alpha$ varies in the interval of Proposition \ref{triangleclassif}.
\end{proof}

\begin{Rema}
	Any representation in this component can be obtained as a deformation of a representation decomposable in products of reflections of order $2$. Indeed, such a representation corresponds to $\theta=\pi$, and $\alpha$ fixed. So any representation corresponding to the same $\alpha$ is obtained by deforming $\theta$.
\end{Rema}

Applying the method of Section \ref{tracesanddecomp}, we can describe all $\RR$-Fuchsian components for $(p,q,r)$-groups. As an example, let us consider the $\RR$-Fuchsian component for the $(3,3,4)$-group. 

\begin{Prop}
	Consider the three angle pairs: $(4\pi/3,2\pi/3),$ $(4\pi/3,2\pi/3),$ $(3\pi/2,\pi/2)$. Then the relative component of $\chi(\Gamma_{3,3,4}^{(2)})$ with respect to these three conjugacy classes contains irreducible representations which are spherical decomposable. It is parameterized by $\alpha,\theta$ with $\theta\in]2\pi/3,4\pi/3[$ and $\alpha$ in the domain defined in Proposition \ref{triangleclassif}.
\end{Prop}

\begin{proof}
	Such a group may be represented by a deformation of the real Fuchsian representation, with $A,B,C$ in the following conjugacy classes:
	$$A,B\sim\begin{pmatrix}
	e^{2i\pi/3} & 0 & 0 \\
	0 & e^{-2i\pi/3} & 0 \\
	0 & 0 & 1 \\
	\end{pmatrix}, \ C\sim\begin{pmatrix}
	i & 0 & 0 \\
	0 & -i & 0 \\
	0 & 0 & 1 \\
	\end{pmatrix}.$$
	Using Theorem \ref{RFuchs}, we know that a spherical decomposition exists, with $\theta_1=\theta_2=\theta_3=\theta$.
	
	Since $(\operatorname{Tr}(A),\operatorname{Tr}(B),\operatorname{Tr}(C))=(0,0,1)$ we obtain
	$$\sin^2(\phi_3)=\sin^2(\phi_1)=\frac{3}{4\sin^2(\theta/2)}, \ \sin^2(\phi_2)=\frac{1}{2\sin^2(\theta/2)}.$$
	Therefore $\theta\in]2\pi/3,4\pi/3[$, and the angular invariant $\alpha$ varies in the interval of Proposition \ref{triangleclassif}.
\end{proof}

 The region depicted in Figures \ref{alphafixe} and \ref{thetafixe} is the set of possible values of $\operatorname{Tr}(A^{-1}B)$. We obtain the curves in Figure \ref{alphafixe} by fixing $\alpha$, and those in Figure \ref{thetafixe} by fixing $\theta$. These drawings illustrate the parameterization of this component of $\chi^{sd}(\Gamma_{3,3,4}^{(2)})$ by the couple $(\alpha,\theta)$. Recall that each value of $\operatorname{Tr}(A^{-1}B)$ corresponds to two distinct representations (see Theorem \ref{charvar} and Corollary \ref{charvarcoro}). The unique reducible point of this component is the one that appears as a singular point on the curve. The horizontal axis corresponds to $\theta=\pi$, i.e.\ to decompositions in products of involutions.

\begin{minipage}{0.4\textwidth}
	\begin{figure}[H]
		\includegraphics[scale=0.2]{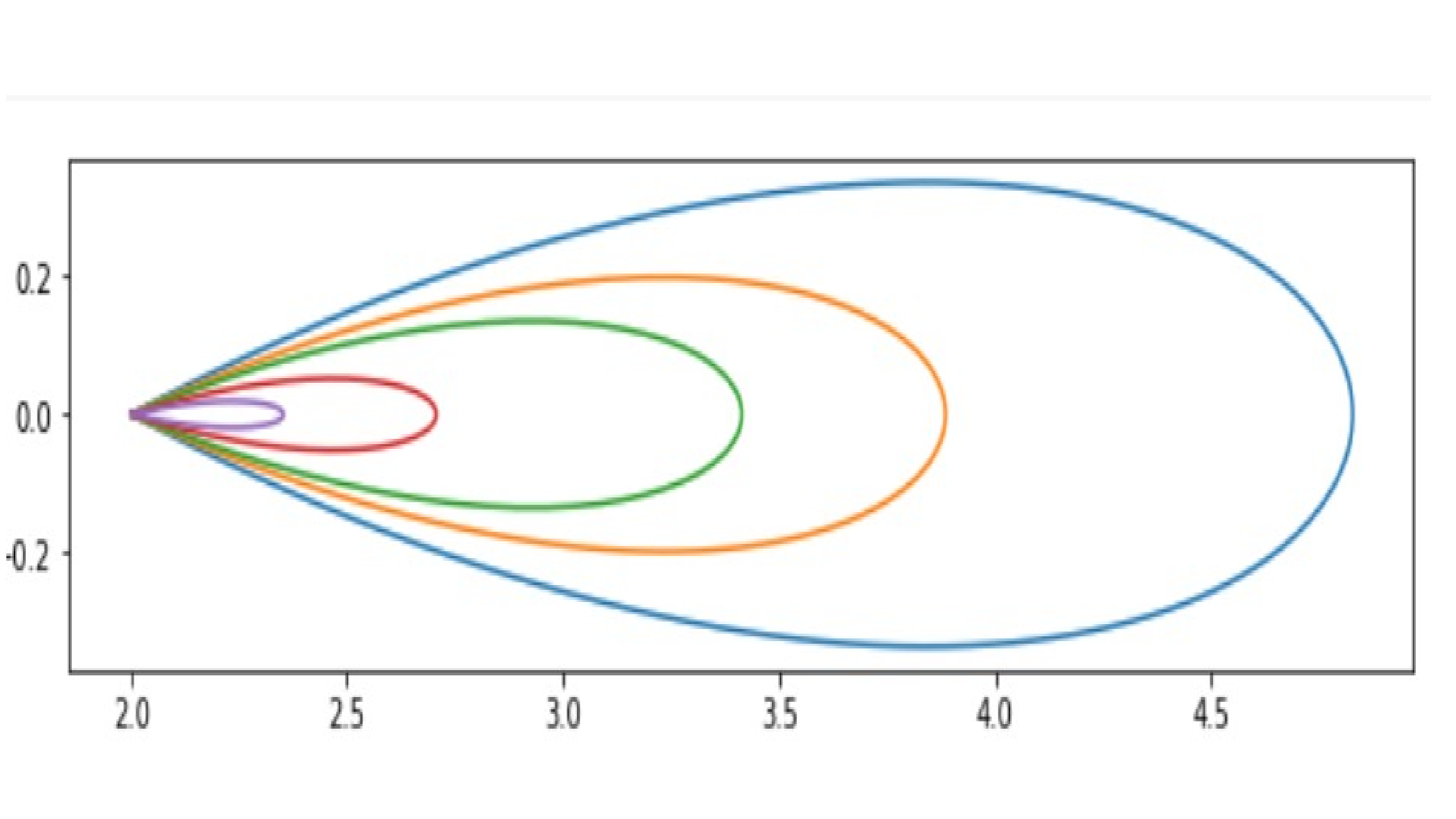}
		\caption{Values of $\operatorname{Tr}(A^{-1}B)$ along the level curves of $\alpha$.}
		\label{alphafixe}
	\end{figure}
\end{minipage}
\hspace{4ex} 
\begin{minipage}{0.4\textwidth}
	\begin{figure}[H]
		\includegraphics[scale=0.2]{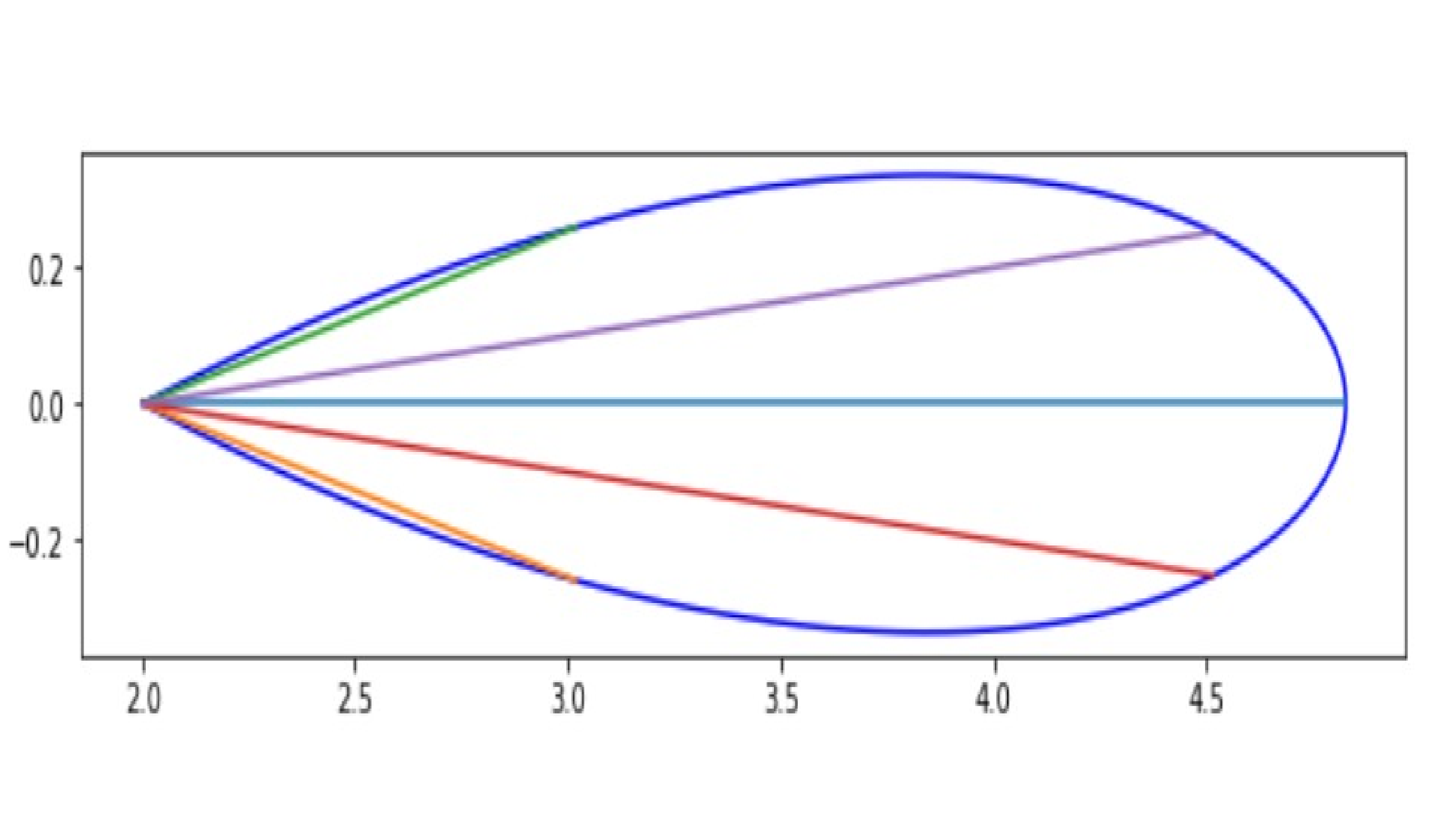}
		\caption{Values of $\operatorname{Tr}(A^{-1}B)$ along the level curves of $\theta$.}
		\label{thetafixe}
	\end{figure}
\end{minipage}

\subsection{An exotic component}

We conclude this paper by giving the parameterization of a non-$\RR$-Fuchsian component of the character variety of the $(4,5,20)$-group, that is, a component containing no $\RR$-Fuchsian representation.

\begin{Prop}
	Consider the three angle pairs $(3\pi/2,\pi),$ $(6\pi/5,2\pi/5),$ $(8\pi/5,3\pi/10)$. Then the relative component of $\chi(\Gamma_{4,5,20}^{(2)})$ with respect to these three conjugacy classes is composed of a unique spherical reducible representation and a family of irreducible spherical decomposable representations parameterized by $\alpha,\theta$ with $\theta\in]\pi,3\pi/2[$ and $\alpha$ in the domain defined in Proposition \ref{triangleclassif}.
\end{Prop}

\begin{proof}
	We represent $A,B,C$ in $\operatorname{SU}(2,1)$ by the following conjugacy classes:
	$$A\sim\begin{pmatrix}
	e^{4i\pi/3} & 0 & 0 \\
	0 & e^{5i\pi/6} & 0 \\
	0 & 0 & e^{-i\pi/6} \\
	\end{pmatrix}, \ B\sim\begin{pmatrix}
	e^{2i\pi/3} & 0 & 0 \\
	0 & e^{-2i\pi/15} & 0 \\
	0 & 0 & e^{-8i\pi/15} \\
	\end{pmatrix},$$
	$$C\sim\begin{pmatrix}
	-1 & 0 & 0 \\
	0 & e^{3i\pi/10} & 0 \\
	0 & 0 & e^{7i\pi/10} \\
	\end{pmatrix}.$$
	The fact that the angle pair of $C^{-1}$ lies on the spherical reducible segment of the image of $\mu$ for the pair $(A,B)$ (see Figure \ref{moment}) guarantees that these conjugacy classes give way to a spherical reducible representation. Thanks to the method of Section \ref{tracesanddecomp}, we now construct explicitely a two-dimensional family of irreducible spherical decomposable representations with respect to these three conjugacy classes.
	
	We first obtain $(\theta_1,\theta_2,\theta_3)=(\theta,\theta-\pi/2,\theta-21\pi/10)$. Then, the formulas giving the expressions of the complex angles in terms of $\theta$ impose $\theta\in]\pi,3\pi/2[.$ Finally, we let $\alpha$ vary in the domain of Proposition \ref{triangleclassif}, and we verify that the reflection group associated to these parameters (as constructed in Lemma \ref{reflnormlemma}) gives way to a $(4,5,20)$-group. Since the products $R_1R_2^{-1},R_2R_3^{-1},R_3R_1^{-1}$ have respective orders $4,5,20$, we have proven that this relative component of $\chi(\Gamma_{4,5,20}^{(2)})$ contains irreducible spherical decomposable representations. To conclude, using Theorem \ref{fiber}, this relative component of $\chi(\Gamma_{4,5,20}^{(2)})$ is composed of a unique spherical reducible representation, and a two dimensional family of irreducible spherical decomposable representations.
\end{proof}

 Just like in Section \ref{334}, we illustrate our parameterization of this component by drawing the values of $\operatorname{Tr}(A^{-1}B)$, for fixed values of $\alpha$ (Figure \ref{alphafixe2}) and fixed values of $\theta$ (Figure \ref{thetafixe2}).

\begin{Rema}
	Relative components containing hyperbolic decomposable representations are more complicated, because the decompositions may include complex reflections in points.
\end{Rema}

\begin{minipage}{0.4\textwidth}
	\begin{figure}[H]
		\includegraphics[scale=0.3]{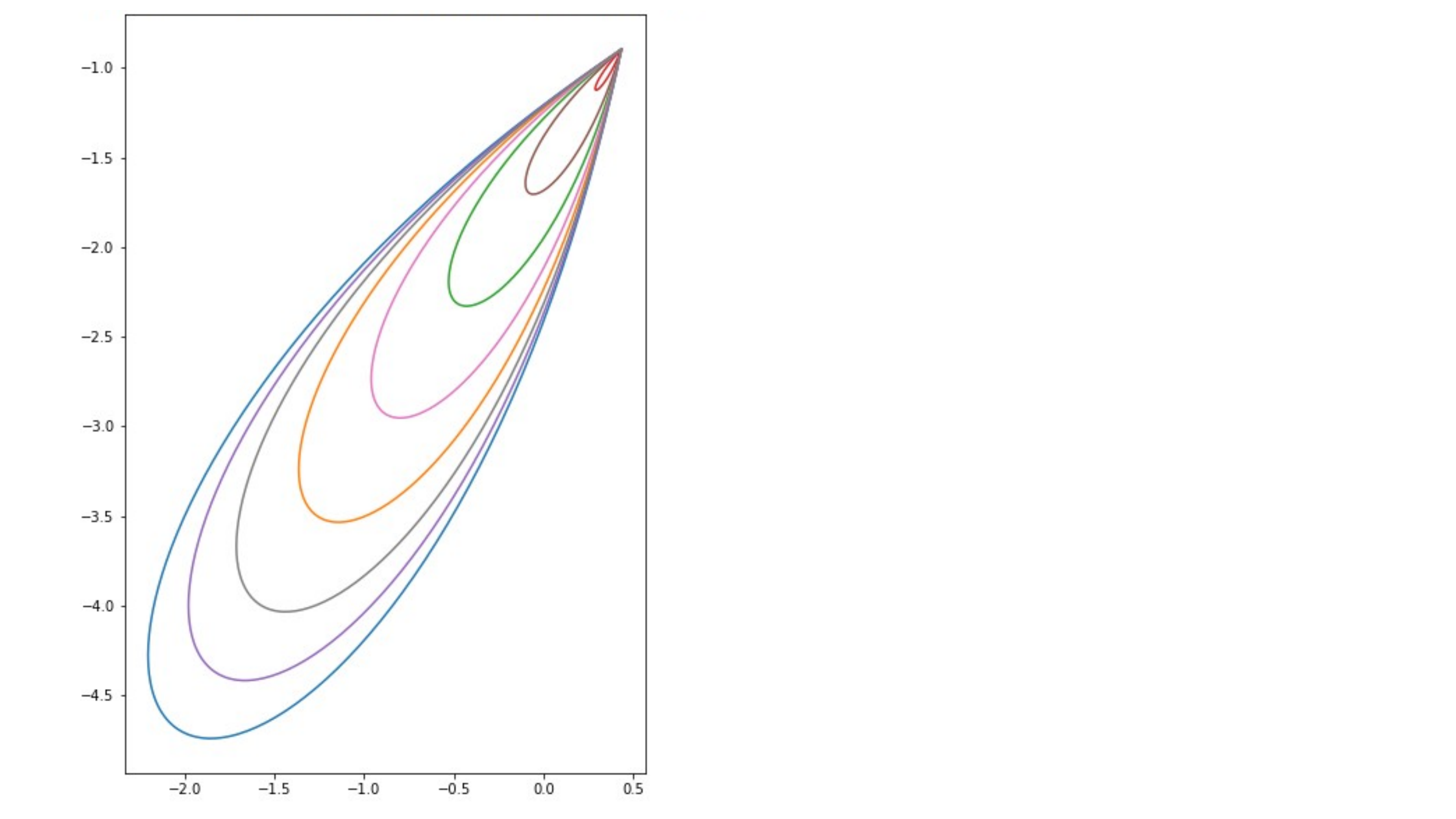}
		\caption{Values of $\operatorname{Tr}(A^{-1}B)$ along the level curves of $\alpha$.}
		\label{alphafixe2}
	\end{figure}
\end{minipage}
\hspace{4ex} 
\begin{minipage}{0.4\textwidth}
	\begin{figure}[H]
		\includegraphics[scale=0.3]{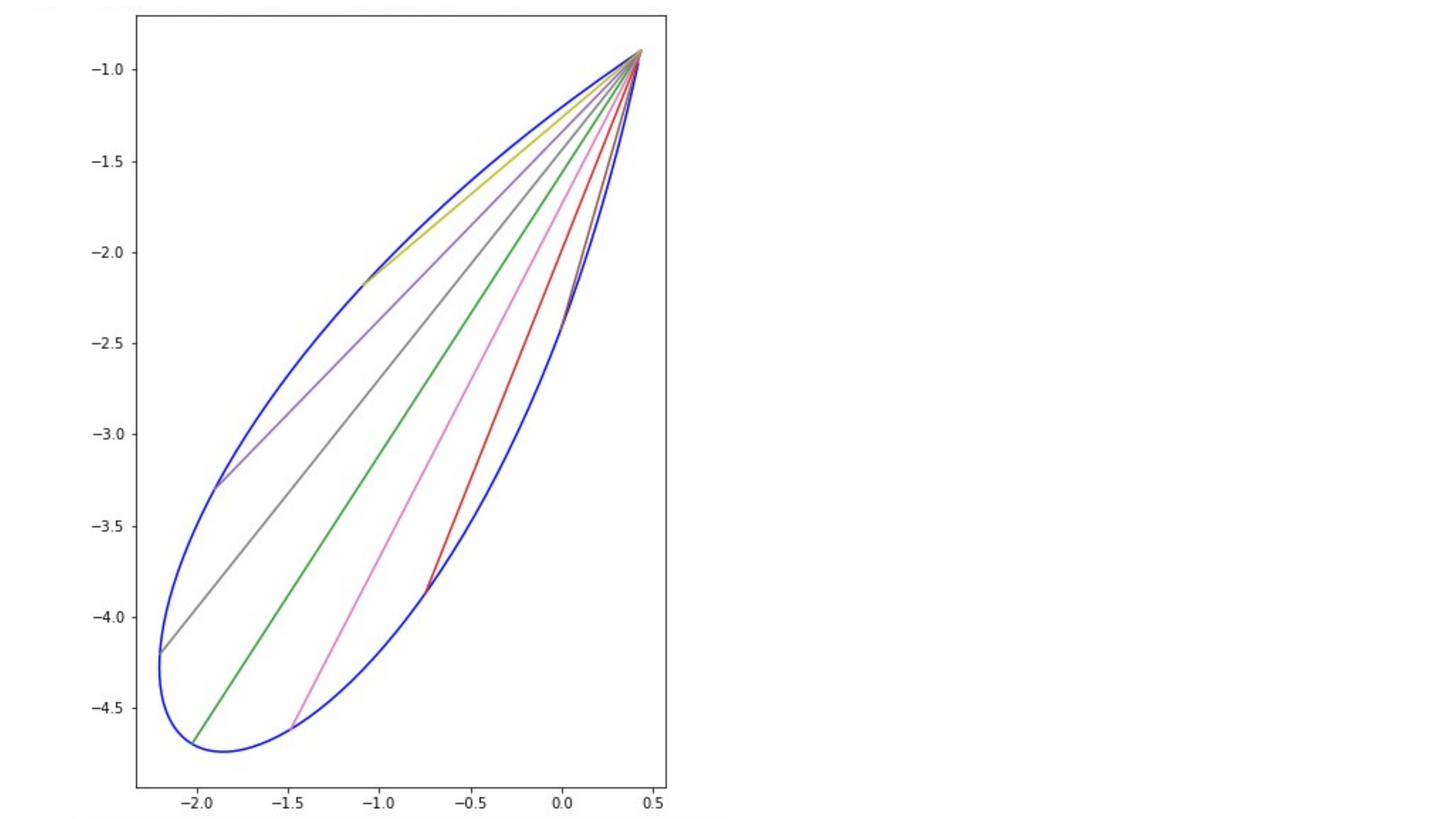}
		\caption{Values of $\operatorname{Tr}(A^{-1}B)$ along the level curves of $\theta$.}
		\label{thetafixe2}
	\end{figure}
\end{minipage}

\bibliography{biblio}
\bibliographystyle{plain}

\end{document}